\documentclass[reqno,oneside,12pt]{amsart}

\usepackage[T1]{fontenc}
\usepackage{times,mathptm}
\usepackage{amssymb,epsfig,verbatim,xypic}
\def\vs{\vspace{0.1cm}}

\theoremstyle{plain}

\newtheorem{thm}{Theorem}[section]

\newtheorem{cor}[thm]{Corollary}
\newtheorem{pro}[thm]{Proposition}
\newtheorem{lem}[thm]{Lemma}
\newtheorem{proposition-principale}[thm]{Proposition principale}
\newtheorem{thm-principal}{Th\'eor\`eme principal}[section]

\theoremstyle{definition}

\newtheorem{eg}[thm]{Example}
\newtheorem{rem}[thm]{Remark}

\newenvironment{thm-A}
{{\vs \noindent \bf Theorem A.$\,$}\it}{\vs}

\newenvironment{thm-*}
{{\vs \noindent \bf Theorem.$\,$}\it}{\vs}

\newenvironment{thm-B}
{{\vs \noindent \bf Theorem B.$\,$}\it}{\vs}

\newenvironment{thm-C}
{{\vs \noindent \bf Theorem C.$\,$}\it}{\vs}

\newenvironment{thm-NB}
{{\vs \noindent \bf Novikov-Boone Theorem .$\,$}\it}{\vs}

\newenvironment{thm-D}
{{\vs \noindent \bf Theorem D.$\,$}\it}{\vs}

\def\C{\mathbf{C}}
\def\R{\mathbf{R}}
\def\Q{\mathbf{Q}}
\def\Z{\mathbf{Z}}

\def\F{{\mathcal{F}}}
\def\G{{\mathcal{G}}}
\def\U{{\mathcal{U}}}

\def\Pic{{\rm{Pic}}}

\def\J{{\text{\sc{j}}}}
\def\L{{\text{\sc{l}}}}

\def\P{\mathbb{P}}

\def\Sing{{\sf{Sing}}}

\def\End{{\sf{End}}}

\def\GL{{\sf{GL}}}
\def\SL{{\sf{SL}}}

\def\Ind{{\text{Ind}}}
\def\Exc{{\text{Exc}}}

\addtocounter{section}{0}             % Start with section 1
\numberwithin{equation}{section}       % Number formulas within sections
\begin{document}

\setlength{\baselineskip}{0.53cm}        % Previous 0.47
%
%%%%%%%%%%%%%%%%%%%%%%%%%%%%%%%%%%%%%%%%%%%%%%%%%%%%%%%%%%%%%%%%%%
%
\title[Pseudo-automorphisms with no invariant foliation]
{Pseudo-automorphisms with no invariant foliation}
%\date{2013}
\author{Eric Bedford}
\thanks{Part of this work was done while E.B. visited Ecole Normale Sup\'erieure (Paris) during the fall of 2012, and he thanks the ENS for its hospitality.} 
\address{Department of mathematics\\ Indiana University\\ Bloomington IN 47405\\ USA,  { \em Current address: } Stony Brook University, Stony Brook, NY, 11794 }
\email{bedford@indiana.edu}
\author {Serge Cantat}
\address{D\'epartement de math\'ematiques et Applications (DMA)\\
         ENS Ulm\\
         Paris, rue d'Ulm\\
         France}
\email{serge.cantat@univ-rennes1.fr, cantat@dma.ens.fr}
\author{Kyounghee Kim}
\address{Department of mathematics\\Florida State University\\Tallahassee FL 32306\\USA}
\email{kim@math.fsu.edu}
%
%%%%%%%%%%%%%%%%%%%%%%%%%%%%%%%%%%%%%%%%%%%%%%%%%%%%%%%%%%%%%%%%%%
%

%
%%%%%%%%%%%%%%%%%%%%%%%%%%%%%%%%%%%%%%%%%%%%%%%%%%%%%%%%%%%%%%%%%%
%

%
%%%%%%%%%%%%%%%%%%%%%%%%%%%%%%%%%%%%%%%%%%%%%%%%%%%%%%%%%%%%%%%%%%
%

%\include{intro-bhps}

\begin{abstract} We construct an example of a birational transformation of a rational threefold for which the first and second dynamical degrees coincide and are $>1$, but which does not preserve any holomorphic (singular) foliation.  In particular, this provides a negative answer to a question of Guedj. On our way, we develop several techniques to study  foliations which are invariant under birational transformations. 
 \end{abstract}

\maketitle

%
%%%%%%%%%%%%%%%%%%%%%%%%%%%%%%%%%%%%%%%%%%%%%%%%%%%%%%%%%%%%%%%%%%
%
%
%%%%%%%%%%%%%%%%%%%%%%%%%%%%%%%%%%%%%%%%%%%%%%%%%%%%%%%%%%%%%%%%%%
%
\section{Introduction}
%
%%%%%%%%%%%%%%%%%%%%%%%%%%%%%%%%%%%%%%%%%%%%%%%%%%%%%%%%%%%%%%%%%%
%

We are interested in smooth complex projective varieties $M$ or, more generally, compact K\"ahler manifolds which 
carry invertible holomorphic or meromorphic transformations $h$ such that 
(i) $h$ has ``rich'' dynamics, and 
(ii) $h$ is not ``integrable''. 

To formalize property (i), we will ask for positive topological entropy ; as explained below, finer requirements can be formulated in terms of dynamical degrees.

 For property (ii), we will say that $h$ is integrable if it preserves a holomorphic (singular) foliation of $M$ of dimension $0 <  d < \dim_\C(M)$; this includes the case of invariant meromorphic fibrations, hence the notion of reducible or non-primitive birational transformations (see \cite{Zhang}, \cite{Cantat:Panorama}). There are several alternative notions of integrability. In \cite{Veselov}, a transformation $f\colon M\dasharrow M$ is integrable if there is a transformation $g\colon M\to M$ such that $g\circ f=f\circ g$ but $f^k\neq g^l$
 for all $(k,l)\neq (0,0)$. In \cite{Malgrange:2001, Malgrange:2002,Malgrange:2004}, tools from differential Galois theory are developed to measure the level of integrability of various dynamical systems; heuristically, a transformation is ``more integrable'' than another if it preserves a richer geometric structure: foliations, volume forms, affine structures, are examples of possible invariant geometric structure in this context (see \cite{casale:2006,casale:2007} for instance). Here, we focus on invariant foliations.
 
 \vspace{0.2cm}
 
 The main goal of this paper is to describe a new family of birational transformations on certain rational threefolds. These examples have three interesting properties: They are pseudo-automorphisms, which means that they are isomorphisms on the complement of Zariski closed subsets of codimension $2$; they are not integrable in the sense that they do not preserve any 
non-trivial foliation; their dynamical degrees exhibit a resonance (namely $\lambda_1(f)=\lambda_2(f)>1$). In particular, these transformations provide 
\begin{itemize}
\item the first examples of pseudo-automorphisms of threefolds which are (proved to be) non-integrable;
\item a negative answer to a question raised by Guedj.
\end{itemize} 

Before giving the precise statement of our results, we describe property (i) and Guedj's question in more detail. On our way, we summarize the main known results in dimension $\leq 2$.

\subsection{Entropy and dynamical degrees}

\subsubsection{} Let $h\colon M\dasharrow M$ be a rational transformation of a complex projective variety or, more generally, a meromorphic
transformation of a compact k\"ahler manifold. It may have indeterminacy points, at which it
does not extend continuously; this indeterminacy set $\Ind(h)$ is a Zariski-closed subset of $M$ of codimension $\geq 2$. 

Let $H^{p,q}(M;\C)$ denote the Dolbeault cohomology groups of $M$. The groups $H^{p,p}(M;\C)$ inherit a natural 
real structure, and the subgroups $H^{p,p}(M;\R)$ contain the cap products of the k\"ahler classes. This is the main reason why
it is sufficient, in what follows, to focus on these cohomology groups (see \cite{Guedj:Panorama}). 

Although $h$ may not  be continuous, it determines a linear operator $h^*_p$ on $H^{p,p}(M;\R)$; however, $(h^*_p)^n$ may
differ from $(h^n)^*_p$ for some values of $n$ and $p$. 
The exponential growth rate of the sequence of linear transformations
$(h^n)^*_p$ is defined by 
\[
\lambda_p(h)= \lim_{n\to +\infty} \parallel (h^n)^*_p\parallel^{1/n}
\]
and is called the {\bf{dynamical degree}} of $h$ of codimension $p$; this real number does not depend on the choice 
of the norm $\parallel \cdot \parallel$ on $\End(H^{p,p}(M;\R))$, and remains invariant if one conjugates $h$ by a birational
map $\varphi\colon M'\dasharrow M$ (see \cite{Guedj:Panorama, Dinh-Sibony}). For $p=0$, one gets $\lambda_0(h)=1$, and for $p=\dim(M)$
the dynamical degree $\lambda_{\dim(M)}(h)$ coincides with the topological degree of $h$. 

We may think of the $p$th dynamical degree as giving the growth rate of cohomology in bidegree $(p,p)$, or volume growth in complex codimension $p$.  The growth of the iterates on the total cohomology group $H^*(M)$  will be dominated by the restrictions to $H^{p,p}(M)$ for $1\le p\le {\rm dim}(M)$.

\subsubsection{} When $h$ is a regular endomorphism of $M$, Gromov and Yomdin proved that the topological entropy
${\sf{h}}_{top}(h)$ of $h\colon M\to M$ coincides with the maximum of the numbers $\log(\lambda_p(h))$, $0\leq p\leq \dim(M)$; Dinh and
Sibony proved that Gromov's upper bound 
\[
{\sf{h}}_{top}(h)\leq \max_{p} \log(\lambda_p(h))
\]
remains valid for dominant meromorphic transformations of compact k\"ahler manifolds (see \cite{Gromov, Yomdin, Dinh-Sibony}). 
Thus, the dynamical degrees provide an upper bound for the complexity of the dynamics of $h$.

One says that $h$ is {\bf{cohomologically hyperbolic}} if the dynamical degrees $\lambda_p(h)$ have a unique maximum. Since
the function $p\mapsto \log(\lambda_p(h))$ is concave, this is equivalent to the condition $\lambda_p(h)\neq \lambda_{p+1}(h)$
for all $0\leq p\leq \dim(M)-1$ (see \cite{Guedj:Panorama}). Thus, cohomological hyperbolicity is a kind of 
non-resonance condition for dynamical degrees. 

\begin{eg}

\noindent{\bf{a.--}} If $f$ is a birational transformation of a projective surface and $\lambda_1(f)>1$, then $f$ is cohomologically
hyperbolic.

\noindent{\bf{b.--}} A paradigmatic example is given by linear endomorphisms of tori. More specifically, consider an elliptic curve $E=\C/\Lambda$, a positive
integer $d$, and the complex torus $A=E^d=\C^d/\Lambda^d$. Let $B$ be a $d\times d$ matrix with integer entries. Since the linear transformation 
$B\colon \C^d\to \C^d$ preserves the lattice $\Lambda^d$, it induces an endomorphism $f_B$ of $A$. The number $\lambda_p(f)$ 
coincides with the square of the spectral radius of $B$ acting on the space $\wedge^p(\C^d)$. For instance, if $B$ is diagonalizable (over $\C$),
with non-zero eigenvalues $\alpha_1$, $\ldots$, $\alpha_d$, then $f_B$ is cohomologically hyperbolic if and only if the moduli $\vert \alpha_i\vert$
are pairwise distinct. 
\end{eg}

Cohomological hyperbolicity has strong dynamical consequences. For instance if either of the following occurs:
\begin{itemize}
\item $\lambda_{\dim(M)}(h)>\lambda_p(h)$ for all $p< \dim(M)$, or 
\item $h$ is an automorphism of a complex projective surface with $\lambda_1(h)>1$ (see also \cite{Bedford-Diller:2005, Dujardin} for birational transformations),
\end{itemize}
then $h$ preserves a unique probability measure $\mu_h$ with entropy $\log(\lambda_{\dim(M)}(h))$ (resp. $\log(\lambda_1(h))$); in particular, 
this number is equal to the topological entropy of $h$; moreover, isolated periodic points of $h$ of period $n$ equidistribute towards
$\mu_h$ as $n$ goes to $+\infty$, and most of them are repelling (resp. saddle, in the second case) periodic points. See \cite{Guedj:Panorama, DNT,Cantat:Milnor}
for this kind of result.

\subsection{Low dimensions} 
In the case of dimension $1$, $M$ is a compact Riemann surface.  When the genus is $\ge2$, 
the automorphism group is finite, and when the genus is $0$ or $1$, an automorphism is essentially 
linear or affine.  Thus the holomorphic dynamics of a single invertible transformation in dimension $1$ is quite 
simple.

Assume now that $M$ is a compact complex surface. If $M$ carries bimeromorphic transformations with positive entropy, then $M$ is bimeromorphically equivalent to (1) a torus, (2) a K3 surface or an Enriques surface, or (3) the projective plane (see \cite{Cantat:Milnor}), and there are strategies to construct examples of automorphisms with positive entropy in all three cases.

If $M$ is a torus, a K3 surface, or an Enriques surface, then all bimeromorphic transformations of $M$ are indeed regular and, as such, are automorphisms of the surface. The so-called Torelli theorem provides a tool to determine the group of automorphisms of such a surface once the Hodge structure is known, but this tool is very hard to use in practice for K3 and Enriques surfaces. The case of $2$-dimensional tori is simpler (see \cite{Ghys-Verjovsky}): A good example is provided by tori of the type $E\times E$, where $E$ is an elliptic curve; the group $\GL_2(\Z)$ always acts by automorphism on such tori. 

If $M$ is rational (case (3)), the situation is more delicate: There are many birational transformations with positive entropy; but deciding whether $M$ carries automorphisms with positive entropy is a difficult task for which there is no general strategy. 

Concerning integrability, Diller and Favre proved that the existence of an invariant fibration is not compatible with positive entropy; Cantat and Favre \cite{Cantat-Favre:2003} showed that if $f$ is an infinite order bimeromorphic transformation of a compact k\"ahler surface $M$ that preserves a foliation then, up to a bimeromorphic conjugacy and finite ramified covers, $f$ comes from a monomial transformation of the plane or an affine transformation of a torus (two cases for which 
there are always invariant foliations). Thus the surfaces which can carry maps with invariant foliations, as well as the maps themselves, are special and explicit.

\subsection{Guedj's question}  {\sl{If $h$ is a birational transformation 
of a compact k\"ahler manifold that is not cohomologically hyperbolic, does it necessarily preserve a non-trivial 
fibration or (singular) holomorphic foliation ?}} (see \cite{Guedj:Panorama}, page 103). \\

The idea behind this conjecture is that a resonance between dynamical degrees should be explained by the existence of an invariant algebraic or analytic structure. 
For instance Gizatullin's Theorem states that a birational transformation $f$ of a surface $S$ which
 is not cohomologically hyperbolic preserves a meromorphic fibration:
there is a rational fibration $\pi\colon S\dasharrow  B$ and an automorphism ${\overline{f}}$ of the curve $B$ such that $\pi\circ f={\overline{f}}\circ \pi$ (see \cite{Diller-Favre:2001} and the references therein).  We shall prove that Guedj's question has a negative answer in dimension $3$; the answer is also negative in dimension $2$ for non-invertible rational transformations (see \S \ref{par:KPR} below).

\vfill
\pagebreak

\subsection{The example} 

\subsubsection{}\label{par:Part0} Let $\J$ be the Cremona involution of $\P^3_\C$, defined in homogeneous coordinates by
\[
\J\colon [x_0:x_1:x_2:x_3]\mapsto [x_0^{-1}:x_1^{-1}:x_2^{-1}:x_3^{-1}].
\]
Given $a$ and $c$ in $\C\setminus\{0\}$, let $\L$ be the projective linear transformation 
\[
\L\colon  [x_0:x_1:x_2:x_3]\mapsto [x_3 : x_0+ax_3 : x_1: x_2+ cx_3].
\]
Let $f_{a,c}$ be the composition $\L\circ \J$. We shall prove that $f_{a,c}$ lifts to a pseudo-automorphism of a rational threefold $X_{a,c}$
after a finite sequence of blow-ups of points if 
\begin{equation}\label{eq:parameters}
\ell a^2 + (\ell+1) ac + \ell c^2=0
\end{equation}
for some $\ell>1$. In what follows, fix such parameters $(\ell ,a,c)$, and denote by $f$ the pseudo-automorphism $f_{a,c}\colon X_{a,c}\dasharrow X_{a,c}$.

\subsubsection{} First, we prove that $f^*$ is {\bf{reversible}}: $(f)^*_1$ and $(f^{-1})^*_1$ are conjugate linear transformations of $H^{1,1}(X_{a,c};\R)$. Since $f$ is a pseudo-automorphism, it satisfies $(f^*_1)^{-1}= (f^{-1})^*_1$, and we deduce from the reversibility and the duality between 
$H^{1,1}(X_{a,c};\R)$ and $H^{2,2}(X_{a,c};\R)$  that $\lambda_1(f)= \lambda_2(f)$. Thus, $f$ is not cohomologically hyperbolic. On the other hand, $\lambda_1(f)>1$ if $\ell >1$.

\begin{thm}\label{thm:main} Let $\ell$ be an integer $\geq 2$. 
Let $f$ be the birational transformation $\L\circ \J$, with parameters $(a,c)$ that satisfy Equation~\eqref{eq:parameters}. 
Then $f$ lifts to a pseudo-automorphism of a rational threefold $X_{a,c}$, obtained from $\P^3_\C$ by a finite
sequence of blow-ups of points. Moreover 
\begin{enumerate}
\item $\lambda_1(f)=\lambda_2(f)>1$;
\item $f$ preserves a hypersurface of $X_{a,c}$ on which it induces a birational transformation which is not conjugate to an automorphism;
\item $f$ is not birationally conjugate to an automorphism of a manifold;
\item Neither $f$ nor any of its iterates $f^n$ for $n>0$, preserves any (singular) foliation of dimension $1$ or $2$; in particular, there is no non-trivial invariant fibration.
\end{enumerate}
\end{thm} 

Thus, the same construction gives, indeed, infinitely many examples. New techniques are required to understand their dynamical behavior 
more precisely. For instance, it is not clear whether the cohomological equality $\lambda_1(f)=\lambda_2(f)$ forces some dynamical resonance: Does $f$ preserve a unique measure of maximal entropy ? Are there unusual equalities between the Lyapunov exponents of such a measure ? (see \cite{Guedj:Panorama})

\vfill
\pagebreak

\subsection{Comments}
 
\subsubsection{} If $X$ is the iterated blow-up of $\P^3$ along a finite sequence of points, then by Truong's Theorem \cite{Truong:2012uq} (see also \cite{Bayraktar-Cantat}) every automorphism has dynamical degree 1 and thus entropy zero.  Thus, the best one could hope is to construct pseudo-automorphisms.

\subsubsection{} In the second appendix, we discuss the existence of invariant fibrations for automorphisms of tori of dimension $3$.   These results are also part of the work on tori  \cite{OT1}, carried out independently by Oguiso and Truong.  In addition, \cite{OT2} constructs a rational threefold which carries an automorphism with positive entropy, but which does not have an invariant fibration.  However,  this automorphism is constructed from a torus automorphism and thus has an invariant foliation.

\subsubsection{}\label{par:KPR} Kaschner, P\'erez, and Roeder found recently a rational transformation $f$ of 
the projective plane $\P^2(\C)$
with dynamical degrees $\lambda_2(f)=\lambda_1(f)$ that does not preserve any foliation \cite{KPR}. Thus, Guedj's question has a negative answer for non-invertible maps in dimension $2$, and for invertible maps in dimension $3$, and a positive answer for invertible maps in dimension $2$.

\subsubsection{} This paper is an expanded version of \cite{Bedford-Kim:2012}, in which the pseudo-automor\-phis\-ms $f_{a,c}$ 
where constructed, their dynamical degrees were computed, and the non-existence of invariant fibration was proved. 

See also \cite{Bedford-Diller-K} for related constructions of pseudo-automorphisms. 

\subsubsection{}  Let $f=L\circ \J$ be the composition of a projective linear transformation $L$ of $\P^3$ with $\J$. 
The birational involution $\J$ has exceptional hyperplanes $\Sigma_k=\{x_k=0\}$, $0\le k\le 3$, and $f$ maps $\Sigma_k$ to the $k$-th column of $L$.  The condition for $f$ to be a pseudo-automorphism is, loosely speaking, the condition that the forward orbit of each $\Sigma_k$ lands on one of the points of indeterminacy $e_j$.  Such maps are called {\it elementary} in \cite{Bedford-Kim:2004}.    When such an elementary map lifts to a pseudo-automorphism, the form of $f_X^*$ is described explicitly in \cite{Bedford-Kim:2004}, from which it satisfies the hypotheses of Lemma 2.1 and thus  $\lambda_1(f) = \lambda_2(f)$. 

While the condition that all the $f$-orbits of the $\Sigma_k$ land on points of indeterminacy is easily stated, it poses a system of equations that is hard to solve computationally.  In \cite{Bedford-Diller-K}, the computational difficulties are circumvented by requiring that $f$ preserves a curve $C$ and by using the restriction $f_{\vert C}$.  While this reduction allows to find solutions, the solutions that are found are quite special.  

In the present paper, the maps $\L\circ J$ that satisfy condition $\ell$ do not belong to any of the families constructed in \cite{Bedford-Diller-K}, but they have a similar motivation. Among the various solutions we found, the $\L$ given above seems to be the easiest to work with.

%
%%%%%%%%%%%%%%%%%%%%%%%%%%%%%%%%%%%%%%%%%%%%%%%%%%%%%%%%%%%%%%%%%%
%
\section{Pseudo-automorphisms : construction}\label{par:Part2} 
%
%%%%%%%%%%%%%%%%%%%%%%%%%%%%%%%%%%%%%%%%%%%%%%%%%%%%%%%%%%%%%%%%%%
%

%%%%%%%%%%%
\subsection{Pseudo-automorphisms in dimension $3$}
%%%%%%%%%%%

Let $M$ be a smooth complex projective variety of dimension $3$. Let $f$ be a birational transformation of $M$. 
The indeterminacy locus $\Ind(f)$ is a Zariski closed subset of $M$ of dimension $\leq 1$ (it may have components 
of dimension $0$ and of dimension $1$). One says that an irreducible subvariety $V$ of $M$ is exceptional for $f$
if $V$ is not contained in $\Ind(f)$ and
\[
\dim(f(V\setminus\Ind(f)))<\dim(V).
\]
Let $p$ be a point of $V\setminus \Ind(f)$. Using local coordinates near $p$ and $f(p)$, $V$ is exceptional if and only
if it is contained in the zero locus of the jacobian determinant of $f$. Thus, the union $\Exc(f)$ of the exceptional subvarieties is
either empty or a subvariety of $M$ of codimension $1$. 

One says that $f$ is a {\bf{pseudo-automorphism}} if $\Exc(f)$ and $\Exc(f^{-1})$ are both empty. 
All pseudo-automorphisms share the following nice properties (see \cite{Bedford-Kim:2013}): 
\begin{itemize}
\item $\Ind(f)$ and $\Ind(f^{-1})$ do not contain isolated points (all their components have dimension $1$);
\item $(f^*)_p^n=(f^n)_p^*$ for all $0\leq p\leq 3$ and for all integers $n\in \Z$; 
\item in particular $(f^*)_p^{-1}=(f^{-1})_p^*$.
\end{itemize}

\begin{lem}
Let $f$ be a pseudo-automorphism of a smooth complex projective threefold $M$. If $f^*_1$ and $(f^*)^{-1}_1$ 
are conjugate linear transformations of $H^{1,1}(M;\C)$ then $\lambda_1(f)=\lambda_2(f)$. 
\end{lem}
\begin{proof}
Since $f^*_1$ and $(f_2^{-1})^*$ are conjugate linear operators with respect to the intersection pairing, i.e.
\[
f^*_1(v)\cdot w=v\cdot (f^*_2)^{-1}(w)\quad \forall (u, v) \in H^{1,1}(M;\R)\times H^{2,2}(M;\R),
\]
one gets $\lambda_1(f)=\lambda_2(f^{-1})$. If $f_1^*$ is conjugate to its inverse, one has $\lambda_1(f)=\lambda_1(f^{-1})$.
The conclusion follows.
\end{proof}

%%%%%%%%%%%
\subsection{A family of examples}
%%%%%%%%%%%

%%%
\subsubsection{Definition}
%%%
As explained in Section~\ref{par:Part0}, we consider the birational transformations $f=\L\circ \J\colon \P^3_\C\dasharrow\P^3_\C$ 
with 
\begin{eqnarray*}
\J\colon [x_0:x_1:x_2:x_3] & \mapsto & [x_0^{-1}:x_1^{-1}:x_2^{-1}:x_3^{-1}], \\
\L\colon  [x_0:x_1:x_2:x_3] & \mapsto & [x_3 : x_0+ax_3 : x_1: x_2+ cx_3].
\end{eqnarray*}
Here, $a$ and $c$ are non-zero complex parameters. The inverse of $\L$ is the projective linear
transformation 
\[
\L^{-1}\colon  [x_0:x_1:x_2:x_3]    \mapsto  [x_1-ax_0 : x_2 : x_3-cx_0 : x_0].
\]

%%%
\subsubsection{Geometry of $\J$}\label{par:geom-J}
%%%
The  indeterminacy locus $\Ind(\J)$
is the union of the six edges
\[
\Sigma_{i,j}=\{x_i=x_j=0\}\subset \P^3_\C
\]
of the tetrahedron $\Delta=\{x_0x_1x_2x_3=0\}$. The exceptional locus $\Exc(\J)$ is the union of the four faces 
\[
\Sigma_i=\{x_i=0\};
\]
each plane $\Sigma_i$ is mapped to the opposite vertex
\[
e_i=\{ x_l=0, \forall l\neq i\}.
\]
Since $\J$ is an involution, it blows up each $e_i$ to the opposite face $\Sigma_i$. 

To describe how $\J$ acts on a neighborhood of an edge of $\Delta$, consider the edge $\Sigma_{0,1}$. 
The family of hyperplanes of $\P^3$ containing the line $\Sigma_{0,1}$ is the family $\{sx_0=tx_1\}$. It is globally invariant under the action of $\J$: The plane $\{sx_0=tx_1\}$ is transformed into the plane $\{tx_0=sx_1\}$; in particular, the plane $\Pi=\{x_0=x_1\}$ is invariant. 

On $\Pi$, $\J$ acts as a standard quadratic involution of $\P^2$, mapping $[x_1:x_1:x_2:x_3]$ to $[x_2x_3:x_2x_3:x_3x_1:x_1x_2]$; the line $\Sigma_{0,1}\subset \Pi$ is mapped to the point $[1:1:0:0]$ and the family of lines $\{s'x_2=t'x_3\}$ is globally invariant. The action of $\J$ 
from $\Pi_{s,t}=\{sx_0=tx_1\}$ to $\Pi_{t,s}=\{tx_0=sx_1\}$ is similar.
Thus, locally, $\J$ transforms the family of planes $\Pi_{s,t}$ containing $\Sigma_{0,1}$ to the family of planes $\Pi_{t,s}$ transverse to $\Sigma_{2,3}$.

Another way to describe the same picture is as follows. Blow-up the line $\Sigma_{0,1}$. The exceptional divisor $E$ is a product $\P^1\times \P^1$ and the blow-down map $\pi$ contracts the fibers $\{\star\}\times \P^1$ of the first projection. The strict transform of the pencil of planes $\Pi_{s,t}$ intersect $E$ on the family of horizontal curves $\P^1\times \{\star\}\subset E$. Similarly, the fibers of $\pi$ can be identified to the intersection of $E$ with the strict transforms of the planes $\{s'x_2=t'x_3\}$. Then, $\J\circ \pi$ maps $E$ onto $\Sigma_{2,3}$: It contracts the horizontal curves $\P^1\times \{\star\}$ to $\Sigma_{2,3}$.

\begin{figure}[t]
\centering\epsfig{figure=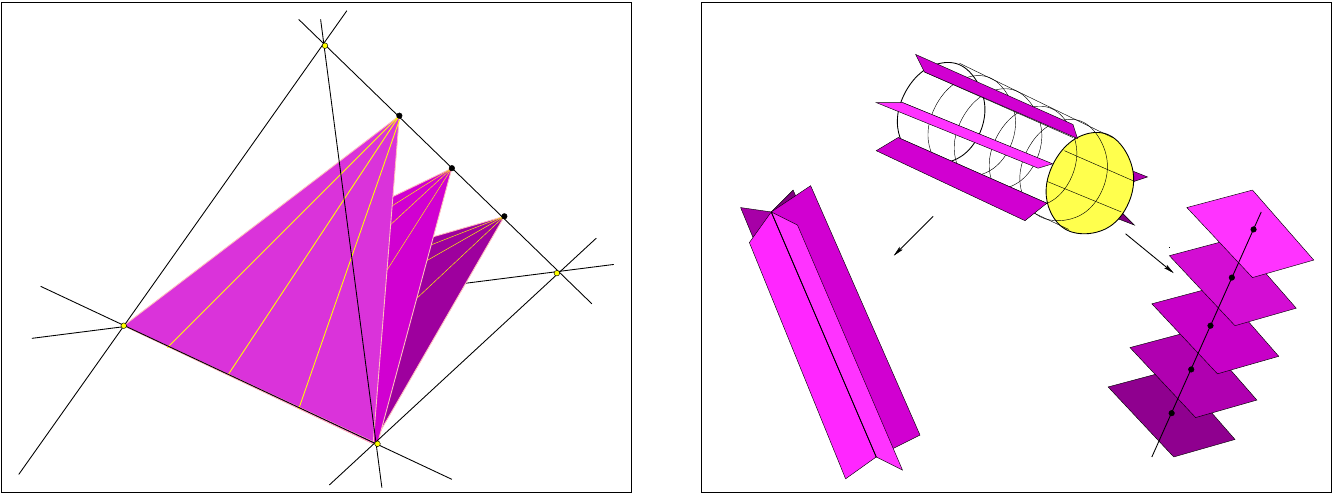} \label{fig:projD_4}
\caption{ {\sc{Action of $\J$}} }
\end{figure}

%%%
\subsubsection{Back to $f$}
%%%
Since $\L$ is an automorphism, $\Ind(f)$ and $\Exc(f)$ coincide with $\Ind(\J)$ and $\Exc(\J)$. On the other hand, 
$\Ind(f^{-1})$ and $\Exc(f^{-1})$ are respectively equal to $\L(\Ind(\J))$ and $\L(\Exc(\J))$; these sets correspond to
the edges and faces of the tetrahedron $\L(\Delta)= \{x_3 ( x_0+ax_3 ) x_1( x_2+ cx_3) =0\}$.

The images of the four planes $\Sigma_i$ under the action of $f$ are the vertices of the tetrahedron $L(\Delta)$; more precisely, 
\[
f(\Sigma_0')=e_1,\, \; f(\Sigma_1')=e_2, \,\; f(\Sigma_2')=e_3, \,\; {\text{and}}\, \; f(\Sigma_3')=p_1,
\]
where 
\[
p_1=[1 : a : 0 : c]
\]
and $f(\Sigma_i')$ stands for $f(\Sigma_i\setminus \Ind(f))$. 
Since $\J$ blows up $e_i$ to $\Sigma_i$, $f$ blows up $e_i$ to $\L(\Sigma_i)$, i.e. 
\[
f\colon e_0\rightsquigarrow \{x_1=ax_0\}, \quad e_1\rightsquigarrow \Sigma_2, \quad e_2 \rightsquigarrow \{x_3=cx_0\}, \quad e_3 \rightsquigarrow \Sigma_0.
\]

%%%
\subsubsection{An intermediate threefold $Y$}
%%%

Blowing up the projective space $\P^3_\C$ at the three points $e_1$, $e_2$, and $e_3$, we get a rational threefold $Y$, together
with a birational morphism $\pi\colon Y\to \P^3_\C$. We denote by $E^Y_1$, $E^Y_2$, $E^Y_3$ the three exceptional divisors: Each $E^Y_j$
is blown down to $e_i$ by $\pi$. Then, we lift $f$ to a birational transformation $f_Y:=\pi^{-1}\circ f \circ \pi$ of $Y$. 

Denote by $\Sigma_j^Y$ the strict transform of $\Sigma_j$ inside $Y$. It is easily verified that $f_Y$ induces a birational transformation 
from $\Sigma_j^Y$ to $E^Y_{j+1}$ for $j=0$, $1$, $2$, and that $f_Y$ blows down $\Sigma_3^Y$ onto the point $\pi^{-1}(p_1)$. 
Similarly, $f_Y$ maps each exceptional divisor $E^Y_j$, $1\leq j\leq 3$, onto the strict transform of the plane $\L(\Sigma_j)$ birationally.  
Thus, the exceptional locus of $f_Y$ coincides with $\Sigma_3^Y$. Moreover $f_Y$ preserves the divisor 
\[
\Gamma^Y:=\Sigma_0^Y \cup E^Y_1 \cup \Sigma_2^Y \cup E^Y_3,
\]
permuting the four irreducible components as follows;
\[
\Sigma_0^Y\dasharrow E^Y_1 \dasharrow \Sigma_2^Y \dasharrow E^Y_3 \dasharrow \Sigma_0^Y.
\]

%%%
\subsubsection{An invariant cycle of curves contained in $\Gamma^Y$}
%%%
Let us consider the restriction of $f_Y$ to the invariant divisor $\Gamma^Y$. We define four curves $\beta_i\subset Y$ as follows. 
\begin{eqnarray*}
\beta_0 & {\text{is the strict transform of}} & \Sigma_0\cap \{ ax_1 = c x_3\}, \\
\beta_1 & {\text{is the strict transform of}}  & E^Y_1\cap \{ax_0=cx_2\}, \\
\beta_2 & {\text{is the strict transform of}}  & \Sigma_2 \cap \{ cx_1=ax_3\}, \\
\beta_3 & {\text{is the strict transform of}}  & E^Y_3 \cap \{cx_0=ax_2\}.
\end{eqnarray*}
The curve $\beta_2$ is the (strict transform) of the line through $e_0$ and $p_1$.
An explicit computation shows that $f_{Y\vert \Gamma^Y}$ permutes the $\beta_i$ cyclically. Moreover, none of these
curves is contained in $\Ind(f_Y)$. 

\begin{eg}\label{eg:Beta}
To describe the kind of calculation that leads to this remark, consider the curve $\beta_2$. It can be parametrized by
$t\in \P^1_\C\mapsto \pi^{-1}\circ \eta(t)$ where 
\[
\eta(t)=[1:t:0:ct/a].
\]
In $\P^3_\C$, the point $f\circ \eta(t)$ coincides with $e_3$. In local coordinates near $E^Y_3$, the blow down map $\pi$
may be written as $(s,u_1,u_2)\mapsto [s:su_1:su_2:1]$. In these coordinates, the equation that determines $\beta_3$, namely $cx_0=ax_2$, corresponds to the
equation $c=au_2$ (once one divides both sides by $s$).
Then, one shows that $f_Y\circ\pi^{-1}\circ  \eta(t)$ corresponds 
to $(0, (a^2+ct)/a, c/a)$ when $t\neq 0$. This is a point of $\beta_3$. 
 \end{eg}

%%%
\subsubsection{The $\ell$-condition}
%%%
The exceptional locus $\Sigma_3^Y$ of $f_Y$ is blown down to the point $p_1\subset \beta_2\subset \Sigma_2^Y$. 
The forward orbit of this point is contained in the invariant cycle of curves $\beta_0\cup\beta_1\cup\beta_2\cup\beta_3$, 
until it reaches an indeterminacy point of $f_Y$ (this may never occur, depending on the values of the parameters $a$
and $c$). We define the sequence $(p_k)$ inductively by $p_{k+1}=f_Y^k(p_{k})$ for all $k\geq 1$ such that $p_1$, $\ldots$, 
$p_k$ does not intersect $\Ind(f_Y)$. If $p_k$ is well defined, one says that the (forward) orbit of $p_1$ is well defined up to 
time $k-1$. For instance, the orbit of $p_1$ is well defined up to time $1$ if $p_1$  is not a point of indeterminacy of $f_Y$. 

Given an integer $\ell\geq 0$, consider the following two steps condition 
\begin{itemize}
\item the orbit of $p_1$ is well defined up to time $4\ell$, 
\item $f_Y^{4\ell}(p_1)=e_0$.
\end{itemize}
This condition is referred to as the {\bf{$\ell$-condition}} in what follows. The $\ell$-condition is just a precise 
formulation of  $f_Y^k(p_1)=e_0$ for some $k>0$, taking into account indeterminacy problems and the fact 
that $k$ must be of the form $4\ell$ for some~$\ell\geq 0$. 

Since $p_1$ is not equal to $e_0$, the $0$-condition is never satisfied. 

The curve $\beta_2$ is parametrized by $t\in \P^1_\C\mapsto \pi^{-1}\circ \eta(t)$ where 
\[
\eta(t)=[1:t:0:ct/a].
\]
The point $p_1$ corresponds to the parameter $t=a$, while $e_0$ corresponds to $t=0$. 
An explicit computation, similar to the one described in Example~\ref{eg:Beta}, shows that the restriction of $f_Y^4$ to $\beta_2$ is induced by
the following translation of the $t$ variable:
\[
h\colon t\mapsto t+ \frac{a^2+c^2 +ac}{c}.
\]
 
\begin{lem}\label{lem:l-cond}
The parameters $(a,c)\in \C^*\times \C^*$ satisfy the $\ell$-condition if and only if
\begin{equation}\label{eq:l-cond}
\ell a^2+(\ell +1) ac + \ell c^2=0.
\end{equation}
\end{lem}

\begin{proof}
The $\ell$ condition means that $f^{4\ell}_Y(p_1)=e_0$; this is equivalent to $h^\ell(a)=0$, hence to 
$\ell a^2+(\ell +1) ac + \ell c^2=0$. The only thing that remains to be shown is that the orbit of $p_1$
is well defined up to time $4\ell$ if the parameters $(a,c)$ satisfy Equation~\eqref{eq:l-cond}. 

For this purpose, let us come back to the computation done in Example~\ref{eg:Beta}. 
One sees that the point of $\beta_2$ corresponding to the parameter $t\neq 0$ 
is mapped to the point of $\beta_3$ with coordinates $(s,u_1,u_2)=  (0, (a^2+ct)/a, c/a)$. Thus, 
if $t\neq 0$ (i.e. $\eta(t)\neq e_0$), the image is not an indeterminacy point of $f_Y$. Hence, 
if $h^k(p_1)\neq e_0$ then $f_Y(h^k(p_1))$ is not an indeterminacy point of $f_Y$.

A similar computation along $\beta_3$, $\beta_0$, and $\beta_1$ concludes the proof. 
\end{proof}

\begin{rem}\label{rem:l-cond}
Assume that $(a,c)$ satisfies the $\ell$-condition. Then 
$\alpha=a/c$ is a root of $\ell \alpha^2+ (\ell+1)\alpha+\ell$. The discriminant of this quadratic 
polynomial is 
\[
\delta_\ell= (\ell+1)^2-4\ell^2=-3\ell^2 + 2 \ell +1.
\]
It is positive for $\ell=0$, vanishes for $\ell=1$, and is negative for $\ell \geq 2$. Thus, $a/c$
is not a real number if $\ell \geq 2$.
\end{rem}

%%%
\subsubsection{The threefold $X$}
%%%
Assume that the $\ell$-condition is satisfied, and blow up the points $p_1$, $p_2$, $\ldots$, $p_{4\ell}$, and $p_{4\ell+1}=e_0$. This
defines a new rational threefold $X$ together with a birational morphism $\tau\colon X\to Y$. The exceptional divisor are
denoted by $P_1$, $\ldots$, $P_{4\ell}$ and $P_{4\ell+1}$. We shall also use the notation $E_0$ for the divisor $P_{4\ell +1}$
because $p_{4\ell+1}=e_0$. By construction, one gets:

\begin{pro}
Let $\ell$ be a positive integer and $(a,c)$ be a pair of non-zero complex numbers that satisfies the $\ell$-condition.
Then, the birational transformation $f_Y$ lifts to a pseudo-automorphism $f_X=\tau^{-1}\circ f\circ \tau$ of the smooth rational 
variety $X$.
\end{pro}

%%%
\subsubsection{Action of $f_X$ on $\Pic(X)$}
%%%

One denotes by $\Sigma_i^X$ the strict transform of $\Sigma_i^Y$ in $X$. Then, by definition,
\begin{itemize}
\item $E_i$, $1\leq i\leq 3$,  is the strict transform of $E_i^Y$ in $X$;
\item ${\hat{E}}_1$ and  ${\hat{E}}_3$ are the total transforms of $E_1^Y$ and $E_3^Y$ respectively. As divisors, 
they are equal to 
\begin{eqnarray*}
{\hat{E}}_1 & = & E_1 + P_4 + P_8 + \ldots + P_{4\ell} \\
{\hat{E}}_3 & = & E_3 + P_2 + P_6 + \ldots + P_{4\ell -2}.
\end{eqnarray*}
\end{itemize}
Let $H$ be the total transform of a hyperplane of $\P^3_\C$ under $\pi\circ \tau$. Then, the classes of $H$, 
${\hat{E}}_1$, $E_2$, ${\hat{E}}_3$, and the $P_j$, $4\ell +1 \geq j \geq  1$, determine a basis of $\Pic(X)$. 
We denote by $\Gamma\subset X$ the strict transform of $\Gamma^Y$: the divisor $\Gamma$ is the sum
$\Sigma_0^X+E_1+\Sigma_2^X+ E_3$. 

\begin{pro}\label{pro:picard-action}
The linear transformations  $f_X^*$ and $(f_X^{-1})^*$ of $\Pic(X)$  are given by
\[
(f_X)^*\,   \colon \left\{ 
\begin{array}{ccl} 
[H] & \mapsto  &  3[H] - 2 [{\hat{E}}_1]  - 2 [E_2] - 2 [{\hat{E}}_3] -2[P_{4\ell+1}] \\ 
  {[{\hat{E}}_1]} & \mapsto & [H] - [{\hat{E}}_1]  - [E_2] - [{\hat{E}}_3]\\ 
  {[E_2]} & \mapsto & [H] - [E_2] - [{\hat{E}}_3] - [P_{4\ell+1}] \\ 
  {[{\hat{E}}_3]} & \mapsto & [H] - [{\hat{E}}_1] - [{\hat{E}}_3] - [P_{4\ell+1}]\\ 
{[P_{j}]} &\mapsto & [P_{j-1}]   \quad\quad  2\leq j\leq 4\ell +1 \\     %\quad \quad({\mathrm{for }} \; 4\ell+1\geq j \geq 2) \\ 
{[P_1]} & \mapsto &  [H] - [{\hat{E}}_1] - [E_2] - [P_{4\ell+1}]
\end{array} \right.
\]
and
\[
(f_X^{-1})^*  \colon \left\{ 
\begin{array}{ccl} 
{[H]} & \mapsto  &  3 [H] - 2 [{\hat{E}}_1]  - 2 [E_2] - 2 [{\hat{E}}_3 ] -2[P_{1}] \\
{[{\hat{E}}_1]} & \mapsto & [H] - [{\hat{E}}_1]  -  [{\hat{E}}_3] - [P_1] \\
{[E_2]} & \mapsto & [H] - [{\hat{E}}_1] - [E_2] - [P_{1}] \\
{[{\hat{E}}_3]} & \mapsto & [H] - [{\hat{E}}_1] - [E_2] -  [{\hat{E}}_3]\\
{[P_j]} & \mapsto & [P_{j+1}]    \quad\quad 1\leq j\leq 4\ell \\                       %\quad \quad({\mathrm{for }} \; 1\leq j \leq 4\ell) \\
{[P_{4\ell+1}]} & \mapsto & [H] - [E_2] - [{\hat{E}}_3] - [P_{1}]
\end{array} \right.
\]
\end{pro}
 
\begin{proof}
Let us start with the basis of $\Pic(X)$ consisting of  $[H]$, the $[E_j]$'s and the $[P_k]$'s, which are the prime blow-up divisors.  Let $B$ be a divisor corresponding to one of these basis elements.  Since $f_X$ is a pseudo-automorphism, $f_X^{-1}$  will be a local diffeomorphism in a neighborhood of the generic point of $B$.  Thus $f_X^{-1}B$ will be an irreducible hypersurface, and we will have  $f_X^*[B] = [f^{-1}_X B]$.   It now remains to express the divisor  $[f^{-1}_X B]$ in terms of our basis.

There are two cases.  First, suppose that $f_X^{-1}B$ lies over a point  $p\in {\bf P}^3$.  Then it will be one of the basis divisors $B'$, and  we have $[f^{-1}_X B] = [B']$.  

The other case is that $f^{-1}_X B$ projects to an irreducible hypersurface $\pi(f^{-1}_X B)\subset {\bf P}^3$ of some degree $\delta$.   We will then have $[f^{-1}_XB]= \delta [H] - \sum \mu_S [S]$, where the sum is taken over all basis elements $[S]$ which project to a point of $\pi(f^{-1}_X B)$.  The multiplicities $\mu_S$ will be  nonnegative integers, which remain to be determined.   The procedure of determining the multiplicity is relatively routine; an example  is given in the Appendix I.

This then gives $f_X^*$ in terms of the basis described in the first line of the proof.  We now change basis, replacing $[E_j]$ by $[\hat E_j]$, and find the formula above.  
\end{proof}

\begin{cor} If the $\ell$-condition is satisfied, then:
\begin{enumerate}
\item The characteristic polynomials of $f_X^*$ and $(f_X^{-1})^*$ are both equal to 
\[
\chi_{\ell}(x)= x^{4\ell+1}(x^4-x^2-x-1)+ x^4+x^3+x^2-1.
\]
\item The factor $(x^4-1)$ divides $\chi_\ell(x)$, and $x=1$ is a simple root of $\chi_\ell$.
\item $\lambda_1(f)=\lambda_2(f)=1$ if $\ell=1$, and $\lambda_1(f)=\lambda_2(f)>1$ is a Salem number if
$\ell \geq 2$.
\end{enumerate}
In particular, the pseudo-automorphism $f_X$ is not cohomologically hyperbolic.
\end{cor}

\begin{proof}
The characteristic polynomial is easily obtained from the previous proposition. It is divisible by
$(x+1)$, and 
\[
\frac{\chi_{\ell}(x)}{x+1}= x^{4\ell+1}T(x)- x^3T(x^{-1})
\]
with $T(x)= x^3-x^2-1$. The cubic polynomial $T$ has one real root $\alpha_T>1$ and two complex 
conjugate roots of modulus $<1$. Thus, $\alpha_T\simeq 1.46$ is a Pisot number of degree $3$. This implies that (i) the largest
root of $\chi_{\ell}$ is a Salem number $\lambda_\ell$ as soon as $4\ell + 1 > 2\frac{\alpha_T+1}{\alpha_T-1}\simeq  10.7$, hence as soon as $\ell \geq 3$, and that (ii) $\lambda_\ell$
converges towards $\alpha_T$ when $\ell$ goes to $+\infty$. (see \cite{Bertin:1989}, Theorem 1.1.1, page 16) 

For $\ell=2$, one checks directly that $\lambda_\ell$ is a Salem number of degree 8, solution of the equation
$
x^8-x^7-x^5+x^4-x^3-x+1=0.
$
\end{proof}

%%%
\subsubsection{Invariant hypersurfaces}\label{par:InvHyper}
%%%

In $\Pic(X)$, the eigenspace of $f^*_X$ for the eigenvalue $1$ coincides with the line
generated by the class of $\Gamma$, i.e. by 
\[
[\Gamma]=2H - {\hat{E}}_1 - E_2 - {\hat{E}}_3 - P_{4\ell+1} - \ldots - P_1.
\] 
It follows that if $S\subset \P^3$ is an $f$-invariant hypersurface and $P$ is a homogeneous equation of $S$,
then $P$ is a polynomial of degree $2m$, for some $m>0$, and vanishes with multiplicity $m$ at each $e_i$; 
moreover, the $f$-invariance is equivalent to 
\[
P\circ F = \lambda ({\rm{Jac}}(F))^m P
\]
for some $\lambda\in \C^*$, where $F$ is the lift of $f$ to $\C^4$ defined by
\[
F(x_0, x_1, x_2,x_3)= (x_0x_1x_2 , \; x_1x_2x_3+ax_0x_1x_2, \;  x_0x_2x_3,  \;  x_0x_1x_3+ cx_0x_1x_2).
\]
 and ${\rm{Jac}}(F)=x_0x_1x_2x_3$ is its
jacobian determinant.

The class $[\Gamma]$ coincides with $-(1/2)[K_X]$, where $K_X$ is the canonical divisor. 
Consider, on the projective space $\P^3_\C$, the rational volume form 
\[
\Omega:=d\left(\frac{1}{x_0}\right)\wedge dx_1 \wedge d\left(\frac{1}{x_2}\right).
\]
This form does not vanish, and its poles are the two faces $\Sigma_0$ and $\Sigma_2$ of the tetrahedron $\Delta$.
Moreover, $\L^*\Omega=dx_2\wedge dx_0\wedge d(1/x_1)$. Hence, 
\[
f^*\Omega= \Omega.
\] 

Defining $\Omega_X=\tau^*(\pi^*(\Omega))$, one gets a rational section of $K_X$ with poles along $\Gamma$
(of multiplicity $2$) and no zeros.

\begin{rem}
Let $M$ be a smooth projective variety and $f$ be a pseudo-auto\-mor\-phism of $M$.
Since the Jacobian determinant of $f$ does not vanish on the complement of $\Ind(f)$, 
the linear operator $f^*$ preserves the canonical class. 
\end{rem}

%%%%%%%%%%%
%%%%%%%%%%%

\section{Behavior on the invariant surface}
%%%%%%%%%%%
%%%%%%%%%%%
\subsection{The action of $f_X$ on $\Gamma$}
%%%%%%%%%%%

Let $g$ be the restriction of $f^4$ to the invariant plane $\Sigma_0$. In coordinates, we have 
\begin{eqnarray*}
g\colon [x_1:x_2:x_3] & \mapsto & [ (cx_1x_2+x_1x_3 + cx_2x_3) \times (cx_1x_2+x_1x_3+a x_2 x_3 +cx_2x_3)  \\
     & &:\;  x_2 x_3 \times (cx_1x_2+x_1x_3 + cx_2x_3) \\
     & &:\;  x_3\times  (ax_2 + cx_2 + x_3) \times (cx_1x_2+x_1x_3+a x_2 x_3 +cx_2x_3)]
\end{eqnarray*}
Hence 
\[ g [x_1:x_2:x_3] = [Q(Q+ax_2x_3) : x_2 x_3 Q : x_3 (x_3+(a+c)x_2) (Q+ax_2x_3)]
\]
where $Q$ denotes the quadratic polynomial function $Q = cx_1x_2+ x_1x_3 + cx_2x_3$. In particular, the degree of $g$ is $4$ (i.e., the pre-image of a line is a curve of degree $4$). 

The inverse of $g$ is the birational map 
\[
g^{-1}\colon [x_1:x_2:x_3]  \mapsto [x_1(x_1-(a+c)x_2)(N-cx_1x_2):x_1x_2 N: N(N-cx_1x_2)]
\]
with $N=x_1x_3-ax_2x_3-ax_1x_2$.

%%%
\subsubsection{Exceptional locus}
%%%

The exceptional locus $\Exc(g)\subset \Sigma_0$ 
is the union of $4$ curves $L_1$, $L_2$, $L_3$, and $L_4$, defined as follows:
\begin{itemize}
\item $L_1$ is the line $\{x_3=0\}$ through the  two points $e_1$ and $e_2$; hence $L_1=\Sigma_{03}$ is an edge
of the tetrahedron $\Delta$; its image under the action of $g$ is $e_1$. 
\item $L_2$ is the line $\{cx_2+x_3=0\}$; it is mapped onto the point $p_3=[1:(a+c)^{-1}:a/c]$ (this point is on $\beta_0$
and is the image of $p_1$ by $f_Y^2$). 
\item $L_3$ is the conic $\{cx_1x_2+x_1x_3+cx_2x_3 =0\}$; its image is the point $e_3$.
\item $L_4$ is the conic $\{ cx_1x_2+x_1x_3+cx_2x_3 +ax_2x_3 =0\}$; its image is $e_2$ (the intersection point of $\beta_0$ and $L_1=\Sigma_{03}$).
\end{itemize}

\begin{rem}\label{rem:exc-loc-g}
In affine coordinates $(x_2,x_3)$ near the point $e_1=[1:0:0]$, the equations of $L_2$, $L_3$, $L_4$ have the same
$1$-jet $cx_2+x_3$. This means that their strict transforms intersect the blow-up $E_1$ of $e_1$ at the same point. 
\end{rem}

%%%
\subsubsection{Indeterminacies}
%%%

The indeterminacy locus of $g$ is 
\[
\Ind(g)= \{e_1, e_2, e_3, p_{4\ell-1}\},
\]
with $p_{4\ell-1}=[c(a+c): -a : a (a+c)]$. These points are equal to the strict transforms of $\{N=0\}$, 
$\{N-cx_1x_2=0\}$, $\{x_1=0\}$, and $\{x_1=ax_2\}$ under the action of $g^{-1}$.

%%%
\subsubsection{The surface $W$}\label{par:surf-W}
%%%
Let us blow up $\Gamma$ at the $\ell+3$ points $e_1$, $e_2$, $e_3$ and $p_3$, $p_7$, $p_{11}$, ..., up to $p_{4\ell-1}$. We obtain a new surface
$W$, together with a birational morphism $\epsilon\colon W\to \Sigma_0$. We denote by $E_i$ the exceptional divisor above 
$e_i$, for $1\leq i \leq 3$, and by $F_j$ the exceptional divisor above $p_{4j-1}$, for $1\leq j\leq \ell$. The transformation $g$
lifts to a birational transformation $g_W$ of $W$, which acts as follows on these curves:
\begin{itemize}
\item $g_W$ maps $E_1$ onto $E_1$ and $E_3$ onto $E_3$;
\item it maps $L_1$ to a point of $E_1$, and $L_3$ to a point of $E_3$. Moreover, the forward orbit
of these points form a sequence of points of $E_1$ (resp. $E_3$) that do not intersect $\Ind(g_W)$. 
\item $g_W$ maps $L_4$ onto $E_2$ and then $E_2$ onto the strict transform of the conic 
$\{ a x_1 x_2 + cx_1 x_2 - x_1 x_3 +ax_2x_3=0\}$.
\item it maps $L_2$ onto $F_1$, then onto $F_2$, ... , onto $F_{\ell}$ and then onto the strict transform
of the line $\{ ax_2-x_1=0\}$. 
\end{itemize}

\begin{proof} Let us prove the two assertions concerning $E_1$ and $L_1$; the remaining ones are proved
along similar lines.

There are local coordinates $(u,v)$ near $E_1$ such that $\epsilon(u,v)=[1:u:uv]$, with $E_1=\{u=0\}$.
In these coordinates, $g\circ \epsilon(u,v)$ is equal to 
\[
 \left[ 1:\; \frac{uv}{c+v+cuv+auv}:\; \frac{uv}{c+v+cuv+auv}\frac{(a+c+v)(c+v+cu+auv)}{c+v+cuv} \right] .
\]
This implies that $g_W$ maps $E_1$ onto $E_1$, as mentioned above, transforming the $v$-coordinate into $v'=v+(a+c)$. 
There is a unique indeterminacy point of $g_W$ on the curve $E_1$; it corresponds to the coordinates $(u,v)=(0,-c)$, and 
coincides with the common intersection point of $E_1$ with the strict transforms of the curves  $L_2$, $L_3$, $L_4\subset\Exc(g)$ 
(see Remark~\ref{rem:exc-loc-g}). The strict transform of the line $L_1$ intersects $E_1$ at the point $(u,v)=(0,0)$; at that point, 
$g_W$ is regular, and behaves like $(u,v)\mapsto (uv/c+{\rm {h.o.t.}}, uv(a+c)/c+{\rm {h.o.t.}})$. 

Moreover, $g[x_1:x_2:x_3]$ is equal to 
\[
\left[ 1:\; \frac{x_2x_3}{Q + a x_2 x_3}: \; \frac{x_3((a+c)x_2+x_3)}{Q} \right]
\]
where $Q=cx_1x_2+x_1x_3+cx_2x_3$. This means that $\pi^{-1}\circ g$ maps the point $[x_1:x_2:x_3]$ to the point $(u,v)$
with 
\[
u= \frac{x_2x_3}{Q + a x_2 x_3}, \quad v= \frac{((a+c)x_2+x_3)(Q+ax_2x_3)}{x_2 Q}.
\] 
Parametrizing the line $L_1$ by $[s:t:0]$, with $[s:t]$ in the projective line, we obtain $Q(s,t,0)=cst$ and
\[
u=0, \quad v= \frac{((a+c)x_2+x_3)(Q+ax_2x_3)}{x_2 Q} = \frac{(a+c)t}{t}=a+c.
\]
This means that the strict transform $L_1'$ is mapped onto the point $(u,v)=(0,a+c)$ of $E_1$. Then, its 
orbit is contained in $E_1$ and corresponds to the infinite sequence of points with coordinates $v_n=n(a+c)$, $n\geq 1$.

Since $(a,c)$ satisfies the $\ell$-condition, we have $\ell a^2 + (\ell+1)ac + \ell c^2=0$; in particular, $a/c$ is not a real 
number (see Remark~\ref{rem:l-cond}). If, for some $n> 0$, $v_n$ coincides with the unique indeterminacy
point $\{v=-c\}$, then $na+(n+1)c=0$ and $a/c$ is the rational number $-(n+1)/n$. Thus, $g_W$ is well defined all along the
positive orbit $(v_n)$. This means that none of the points $g_W^n(L_1')$, $n\geq 1$, is indeterminate; thus, if algebraic stability
fails for $g_W$, this is not because $L_1'$ is eventually mapped to an indeterminacy point.    \end{proof}

Similarly, one proves that $g_W^{-1}$ preserves the curve $E_1$, acts by a translation $v\mapsto v-(a+c)$ on $E_1$, and has
a unique indeterminacy point on $E_1$, namely $v=a+c$. In particular, the backward orbit of the curve $L_1$ is a sequence of curves that intersect $E_1$ along the points of parameters $-n(a+c)$. This implies, that the curves $g^{n}(L_1)$, $n\geq 1$ are pairwise distinct, and that 
none of them is contained in $\Ind(f^{\pm k})$, $k\leq 4$.

%%%
\subsubsection{Action on $\Pic(W)$}
%%%
Recall that a birational transformation $h$ of a surface $Z$ is {\bf{algebraically stable}} if $(h^*)^n=(h^n)^*$ on $\Pic(Z)$
for all integer $n$ or, equivalently, $h^n$ does not contract a curve of $\Exc(h)$ on an indeterminacy point of $h$
for all $n\geq 1$ (see \cite{Diller-Favre:2001}).

To describe the action of $g_W$ on the Picard group of $W$, we make use of the basis given by the class $L$
of (the total transform of) a line, and the classes of the exceptional divisors $E_i$ and $F_j$. We denote by $L_i'$ 
the strict transform of $L_i$ in $W$.  From the previous description of the action of $g_W$, and similar properties
for $g_W^{-1}$, one gets the following lemma.

\begin{lem}\label{lem:g-on-pic}
The transformation $g_W\colon W\dasharrow W$ is algebraically stable. Its action on $\Pic(W)$ is given by 
\[
(g_W)^*\colon\left\{ 
\begin{array}{ccl}
{[L]} & \mapsto & 4 [L]- 2 [E_1] -2[E_2] -[E_3] - [F_\ell] \\
{[E_1]} & \mapsto & [L_1'] + [E_1] = [L]-[E_2] \\
{[E_2]} & \mapsto & [L_4'] = 2[L] - [E_1] -[E_2] -[E_3] \\
{[E_3]} & \mapsto & [L_3'] + [E_3] = 2[L] - [E_1] - [E_2] - [F_\ell] \\
{[F_j]} & \mapsto & [F_{j-1}], \; {\mathrm{ for }} \; \ell\geq j \geq 2\\
{[F_1]} & \mapsto & [L_2']=[L] - [E_1] 
\end{array}
 \right.
\]
\end{lem}

%%%
\subsubsection{$g_W$-invariant curves}
%%%
The eigenspace of $(g_W)^*$ with eigenvalue $1$ is span\-ned by $[\Sigma_{0,2}] = [L] -[E_1]-[E_3]$ and $[\beta_0] =[L]-[E_2] - \sum_{j=1}^\ell [ F_j]$. The indeterminacy locus of $g_W^{-1}$ is $\{ \hat e_1= g_W(L_1') \in E_1, \hat e_3= g_W(L_3')\in E_3\}$. It follows that if $C$ is $g_W$-invariant, then 
\[
(g_W)^* [C] = [C] + \kappa_1 [ L_1'] + \kappa_3 [L_3']
\]
where $\kappa_1, \kappa_2$ are the multiplicities of $\hat e_1$ and $\hat e_3$ in $C$. From the above equation, we see that there is a solution if and only if $\kappa_1 + \kappa_3=0$. Since both $\kappa_1$ and $ \kappa_3$ are non-negative integers, we have $\kappa_1=\kappa_3=0$ and $[C] =m_1 [\Sigma_{0,2}]+m_2 [\beta_0]$; therefore
\[
[C].[\Sigma_{0,2}] = m_2-m_1 \qquad \text{and} \qquad [C].[\beta_0]= m_1-\ell\, m_2.
\]
If none of the irreducible components of $C$ is equal to $\Sigma_{0,2}$ or $\beta_0$, we have $m_2\ge m_1 \ge \ell\, m_2$, and we get a contradiction with $\ell\ge 2$. Thus,  we obtain the following lemma.

\begin{lem}\label{lem:gw-inv}
Any $g_W$-invariant curve $C$ is contained in $\Sigma_{0,2} \cup \beta_0$.
\end{lem}

%%%
\subsubsection{Dynamical degree}
%%%
Assuming, as above, that $(a,c)$ satisfies the $\ell$-condition, one gets the following description of the characteristic 
polynomial of $g_W^*$. 

\begin{pro}\label{pro:g-not-auto}
The dynamical degree $\lambda_1(g)$ of $g=f^4_{\Sigma_0}$ is the largest root of the polynomial 
\[
x^\ell -x^{\ell-1}- x^{\ell-2}- \ldots - x - 1.
\]

If $\ell = 1$, then $\lambda_1(g_W)=1$ and the degree growth of $g$ is linear.
If $\ell\geq 2$ then $\lambda_1(g_W)>1$ is a Pisot number.

In both cases, $g_W$ and its iterates $g_W^n$, $n\neq 0$, are not
birationally conjugate to an automorphism of a rational surface.   
\end{pro}

\begin{proof}
Since $g_W$ is algebraically stable, its dynamical degree is equal to the largest root of the characteristic polynomial of $g_W^*$.
From Lemma~\ref{lem:g-on-pic} we see that the characteristic polynomial is 
\[
x^2(x-1)^2(x^\ell -x^{\ell-1}- x^{\ell-2}- \ldots - x - 1).
\]
If $\ell=1$, the   roots of this polynomial are equal to $0$ or $1$; hence, $\lambda_1(g)=1$.
Moreover, the matrix of $g^*_W$ contains a Jordan block of size $2$ with eigenvalue $1$; it 
follows that the sequence $\parallel (g_W^n)^*\parallel$ grows linearly with $n$. Hence, $g_W$ 
is  conjugate to a Jonqui\`eres transformation and neither $g_W$ nor any iterate $g_W^n$, $n\neq 0$, is conjugate to an automorphism (see \cite{Diller-Favre:2001}).

If $\ell \geq 2$, the unique eigenvalue $\lambda_\ell>1$ is a root of 
\[
\chi_g(t)=x^\ell -x^{\ell-1}- x^{\ell-2}- \ldots - x - 1
\]
This is not a reciprocal polynomial. Hence $g$ is not conjugate to an automorphism. Indeed, if $h$ is an automorphism of a projective surface $Z$, then $h^*$ preserves the intersection form and the ample cone on the N\'eron-Severi group; since this intersection form is non-degenerate, of signature $(1,m)$,  the characteristic polynomial of $h^*$  is a reciprocal polynomial.

More precisely, if $\ell \geq 3$, then $\lambda_\ell$ is a Pisot number of degree $\geq 3$. As such, all its
powers $\lambda_\ell^n$, $n\geq 1$, are Pisot numbers of degree $\geq 3$. Hence, they are not Salem or
quadratic integer and, as such, cannot be the dynamical degree of an automorphism. 

If $\ell=2$, then $\lambda_2= (1+\sqrt{5})/2$ is the golden mean and $\lambda_2^{2}=(3+\sqrt{5})/2$ is reciprocal. 
Thus, a priori, $g^2$ could be birationally conjugate to an automorphism. On the other hand, the 
eigenvectors of $g^*_W$ corresponding to  the eigenvalue $\lambda_2$ are all proportional to 
\[
u:=(2+\sqrt{5})[L]-(3/2+ \sqrt{5}/2) [E_1] - (1/2 +\sqrt{5}/2)[E_2+E_3+F_2]-[F_1], 
\]
and one verifies that $u\cdot u>0$. This implies that $g_W^n$ is not birationally conjugate to an automorphism
for any $n\neq 0$ (see \cite{Diller-Favre:2001}, Theorem 0.4).
\end{proof}

\begin{cor}
No iterate $f_X^n$, $n\neq 0$, of the pseudo-automorphism $f_X$ of $X$ is conjugate, by any birational transformation $\varphi\colon Z\dasharrow X$, to an automorphism of a smooth rational threefold $Z$.
\end{cor}

\begin{proof}
Since $g^n$ is not birationally conjugate to an automorphism, Corollary 1.6 of \cite{Bedford-Kim:2013} shows that $f^{4n}$ is not
birationally conjugate to an automorphism.  
\end{proof}

%%%
\subsubsection{No $g$-invariant foliation}
%%%

\begin{thm}\label{T:no_g_inv}
The surface $W$ does not carry any $g^k$-invariant foliation of dimension $1$ if $k\neq 0$. 
\end{thm}

\begin{proof} 
We refer to Section~\ref{par:InvFoliationsDef} for the definitions concerning foliations, and their invariance under 
the action of a birational mapping. Assume that a non-trivial iterate of $g$ preserves a foliation $\F$.

Since $\ell \geq 2$, we know that $\lambda_1(g)>1$. The classification theorem obtained by Cantat and Favre (see \cite{Cantat-Favre:2003}) implies that a non-trivial iterate $g^m$ is conjugate to one of the following: 
\begin{itemize}
\item[(a)] a monomial transformation $(x,y)\mapsto (x^a y^b, x^c y^d)$ where $a$, $b$, $c$, and $d$ are the coefficients
of an element of $\GL_2(\Z)$;
\item[(a')] a quotient of case $(a)$ by the involution $(x,y)\mapsto (x^{-1}, y^{-1})$;
\item[(b)] a Kummer example, i.e. an automorphism $h$ of a rational surface which "comes from" a linear automorphism of a torus $\C^2/\Lambda$ after a finite cover and a contraction of finitely many invariant exceptional divisors. 
\end{itemize}
In cases (a) and (a'), the dynamical degree is a quadratic integer. In case (b), $g^m$ would be conjugate to an automorphism. 
Thus, Proposition~\ref{pro:g-not-auto} provides a contradiction when $\ell \geq 3$. 

Now let us consider the case $\ell=2$.   We have seen that $g$ is not conjugate to an automorphism, so we need to consider only cases (a) and (a'); we assume that property (a) is satisfied: There exists a monomial map $M\colon \P^2\dasharrow \P^2$ and 
a conjugacy $\psi\colon \P^2\to W$ such $\psi \circ M=g^m \circ \psi$. (case (a') is dealt with similar arguments).
Since $\ell=2$, the dynamical degree $\lambda_1(g)$ is the golden mean $\mu=(1+\sqrt{5})/2$. This implies that $M$ is a hyperbolic element of $\GL_2(\Z)$: its spectral radius $\lambda_M$ is a quadratic integer $>1$; this the dynamical degree is invariant under conjugacy, and $\lambda_M$ is the dynamical degree of $M$, one gets $\lambda_M=\mu^m$. 

If $q$ is a periodic point of $M$ of period $k$, the eigenvalues of the differential $D(M^k)_q$ has two eigenvalues, $\mu_1=\pm \mu^{mk}$ and $\mu_2= \pm (\mu')^{mk}$, with $\mu'= (1-\sqrt{5})/2$. In particular, these eigenvalues are not equal to $1$; more precisely $\mu_1^{n_1}/\mu_2^{n_2}\ne1$ for all pairs of positive integers $(n_1,n_2)$. We shall use this property to
derive a contradiction. For this, we start by an easy observation about blow-ups.

Suppose that $p$ is a fixed point of a local surface diffeomorphism ~$h$.  Suppose that $(X_1,X_2)$ is local coordinate system with $p= X_1\cap X_2$ and that $Dh_p$ is diagonal, with multiplier $\mu_j$ along the $X_j$-axis.
\begin{lem}\label{lem:bup-fpoints}
Let $Z$ denote the space obtained by blowing up the point $p$, and let $E_p$ denote the exceptional blow-up divisor.  Then $h$ lifts to a diffeomorphism $\hat h$ of a neighborhood of $E_p$ inside $Z$,  $E_p\cap X_j$ is a fixed point of $\hat h$, and the local multpliers at this point are $\mu_j$ and $\mu_i/\mu_j$, where $i\ne j$.
\end{lem}
Blow up one of the fixed points of $E_p\cap X_j$.  Continue this way and blow up fixed points over $p$ several times.  Again let $\hat h$ denote the resulting diffeomorphism.  If $E'$ and $E''$ are exceptional blow-up divisors, the intersection points of the form $E'\cap E''$ will be fixed by $\hat h$.  The following is an easy observation about the induced maps on blowups:
\begin{lem}
The diffeomorphism $\hat h$ fixes the blow-up divisors $E'$ and $E''$.  The multipliers of $\hat h$ at $E'\cap E''$ are of the form $\mu_i^{n_i}/\mu_j^{n_j}$ for positive integers $n_1,n_2$.  If the multiplier at the fixed point(s) of the restricted map ${\hat{ h}}_{\vert {E'}}$ is not equal to $1$, then ${\hat{ h}}_{\vert {E'}}$ has exactly two fixed points.
\end{lem}

Now to continue the proof of the Theorem, we recall that the only invariant curves of $g^m$ are $\Sigma_{0,2}$ and $\beta_0$.  The curve $\Sigma_{0,2}$ consists of fixed points, all of which are parabolic; and the only fixed point on $\beta_0$ is $\beta_0\cap\Sigma_{0,2}$. Let $\varphi$ denote $\psi^{-1}$. Two cases may occur. 

First, assume that $\varphi$ is biregular from a neighborhood ${\mathcal{U}}$ of a point $p\in \Sigma_{0,2}$ to a neighborhood ${\mathcal{V}}$ of the point $q=\varphi(p)$; changing $p$ into a nearby point if necessary, we can assume that $p$ and $q$ 
are not indeterminacy points of $g^m$ and $M$. Then, $q$ is a fixed point of $M$, and $D\varphi_p$ conjugates $Dg^m_p$ to $DM_q$; this contradicts the fact that $p$ is a parabolic fixed point of $g^m$ while all fixed points of $M$ are saddle points. 

The alternative is that $\varphi$ blows down $\Sigma_{0,2}$ to an indeterminacy point $q$ of $\psi$. Let $\varepsilon\colon Z\to \P^2$ be a sequence of blow-ups of points above $q$ such that $\psi \circ \varepsilon$ is regular in a neighborhood of $\varepsilon^{-1}(q)$. Denote by $\hat{M}$ the transformation $\varepsilon^{-1}\circ M\circ \varepsilon$ of the surface $Z$. 
Let $C$ be an irreducible component of the exceptional divisor of $\varepsilon$ which is mapped to $\Sigma_{0,2}$
by $\psi \circ \varepsilon$. Then, $\psi \circ \varepsilon$ is biregular in a neighborhood of the generic point of $C$.
If $q$ is not an indeterminacy point of $M$,  one gets
a contradiction because $\psi \circ \varepsilon$ should conjugate $g^m$ to ${\hat{M}}$ and, by Lemma~\ref{lem:bup-fpoints}, ${\hat{M}}$ has no parabolic fixed point on $C$. If $q$ is an indeterminacy point of $M$, one easily verifies that $C$ cannot be 
${\hat{M}}$-invariant, because $M$ is a hyperbolic monomial mapping. Again, one gets a contradiction. 

Case (a') is dealt with similarly. 
\end{proof}

%
%%%%%%%%%%%%%%%%%%%%%%%%%%%%%%%%%%%%%%%%%%%%%%%%%%%%%%%%%%%%%%%%%%
%
\section{No invariant foliations: Foliations by curves}\label{par:fol-curves} 
%
%%%%%%%%%%%%%%%%%%%%%%%%%%%%%%%%%%%%%%%%%%%%%%%%%%%%%%%%%%%%%%%%%%
%

In this section we prove that the non-trivial iterates of $f_X$ do not preserve any foliation of $X$ of dimension $1$.

%%%%%%%%%%%
\subsection{Invariant foliations}\label{par:InvFoliationsDef}
%%%%%%%%%%%

Let $M$ be a complex projective variety. A foliation $\F $ of dimension $1$ on $M$ is given by an open cover $\{{\mathcal{U}}_i\}$  of $M$ and
holomorphic vector fields $V_i$ on each ${\mathcal{U}}_i$ such that
\begin{itemize}
\item[(i)] on the intersections ${\mathcal{U}}_i\cap {\mathcal{U}}_j$, the vector
fields $V_i$ and $V_j$ determine the same tangent directions: There is a holomorphic function $\alpha_{i,j}\colon {\mathcal{U}}_i\cap {\mathcal{U}}_j\to \C^*$ such that
\[
V_i=\alpha_{i,j}V_j.
\]
(the functions $\alpha_{i,j}$ form a cocycle, i.e. an element of the Cech cohomology group $H^1(M,{\mathcal{O}}_M^*)$)
\item[(ii)] the vector fields do not vanish in codimension $1$; in other words, the {\bf{singular locus}} $\Sing(\F)$ of the foliation, which is the set of points locally defined by the equation $V_i=0$, has codimension $\geq 2$. 
\end{itemize}
The {\bf{local leaves}} of $\F$ are defined in ${\mathcal{U}}_i\setminus \Sing(\F)$ by integration of the vector field $V_i$; property (i) shows that they glue together to define the {\bf{global leaves}} of $\F$ (in the complement of $\Sing(\F)$). An algebraic curve $C\subset M$ is called an {\bf{algebraic leaf}} if the defining vector fields $V_i$ are tangent to $C$; thus, an algebraic leaf may include singularities of $\F$ and it may even be contained in the singular locus of $\F$.

If the second assumption is not satisfied, one can divide each $V_i$ by a local equation $f_i$ of the codimension $1$ part of $\Sing(\F)$; one 
then gets a new family of vector fields $V_i'$ and a new cocyle $\alpha_{i,j}(f_j/f_i)$, but the leaves of $\F$ remain locally the same (see below).

Let $f$ be a birational transformation of $M$. One says that $f$ preserves the foliation $\F$ if $f$ maps local leaves to local leaves. More
precisely, for $x$ in ${\mathcal{U}}_i\setminus \Ind(f)$ and ${\mathcal{U}}_j$ an open set that contains $f(x)$, one wants $Df_x(V_i(x))$ to be proportional to $V_j(f(x))$. According to property (i), this does not depend on the choices of the open sets ${\mathcal{U}}_i$ and ${\mathcal{U}}_j$.

%%%%%%%%%%%
\subsection{Action of $\J$ on foliations}
%%%%%%%%%%%

Recall that $\J$ flips $\Sigma_{0,1}$ into $\Sigma_{2,3}$, as described in Section~\ref{par:geom-J}. The following lemma describes the
action of $\J$ on foliations of dimension $1$ that are transverse to the edge $\Sigma_{0,1}$ of the tetrahedron $\Delta$. By symmetry, 
the action of $\J$ is the same near all six edges $\Sigma_{i,j}$ of $\Delta$.

\begin{lem}\label{par:JonFol}
If $\F$ is a one-dimensional foliation near a point $q$ of $\Sigma_{0,1}\setminus\{e_2,e_3\}$ that is transverse to $\Sigma_{0,1}$, 
its image under the action of $\J$  determines a foliation of dimension $1$ that is tangent to $\Sigma_{2,3}$ (i.e. $\Sigma_{2,3}$
is an algebraic leaf of $\J^*\F$).
\end{lem}

\begin{proof}
Consider $\J$ as a birational mapping from the open set $x_3\neq 0$ to the open set $x_0\neq 0$. We set $x_3=1$ in the first open
set and denote the coordinates by $[x_0:x_1:x_2:1]$; in the second open set, we set $x_0=1$ and use coordinates $[1:u:v:w]$. The map
$\J$ can be written as
\[
\J[x_0:x_1:x_2:1]=[1:\frac{x_0}{x_1}:\frac{x_0}{x_2}:x_0]
\]
and its differential is 
\[
D\J_{(x_0,x_1,x_2)}=\left( \begin{array}{ccc} 1/x_1 & -x_0/x_1^2 & 0 \\ 1/x_2 & 0 & -x_0/x_2^2 \\ 1 & 0 & 0 \end{array}\right).
\]
The pre-image of a point $[1:u:v:w]$ is the point with coordinates $[w:w/u:w/v:1]$. Thus, given a local holomorphic vector 
fields $V$ near a point $q$ of $\Sigma_{0,1}$ in the open set $x_3\neq 0$, 
\[
V(x_0,x_1,x_2)=a\partial_{x_0}+b\partial_{x_1}+c\partial_{x_2},
\]
one gets 
\begin{eqnarray*}
\J_*V(u,v,w) & = & \left(\frac{u}{w}a(w,w/u,w/v)-\frac{u^2}{w}b(w,w/u,w/v)\right)\partial_u \\
 & + &  \left(\frac{v}{w}a(w,w/u,w:v)-\frac{v^2}{w} c(w,w/u,w/v)\right)\partial_v \\
 & + & a(w,w/u,w/v)\partial_w.
\end{eqnarray*}
We want to understand what happens to the foliation defined by $\J_*V$ when $(u,v,w)$ approaches
the line $\Sigma_{2,3}$, i.e. when $v$ and $w$ vanish simultaneously. For this, one needs, first, to multiply
$\J_*V$ by a holomorphic function that vanishes along the poles of $\J_*V$, so as to kill the poles of this meromorphic
vector fields. 

By assumption, $V$ is transverse to $\Sigma_{0,1}$ near the point $q$: Which means that the functions $a$ and $b$ do not
vanish simultaneously along $\Sigma_{0,1}$ in a neighborhood of $q$. Assume that $a(q)\neq 0$ (the 
case $b(q)\neq 0$ is dealt with similarly). Then, $a$ is a unit in a neighborhood of $q$, and dividing $V$
by $a$, we can (and do) assume that $a$ is a non-zero constant in a neighborhood of $q$. Thus, 
$\J_*V$ has a pole along $w=0$, and the local holomorphic vector fields defining $\J_*\F$ are multiples
of 
\[
w\J_*V =  \left(ua -u^2b(w,w/u,w/v)\right)\partial_u + \left(v a -v^2 c(w,w/u,w/v)\right)\partial_v  +  wa \partial_w.
\]
Either $wJ_*V$ is already holomorphic, or we have to multiply it by a function $\varphi(u,v,w)$ that vanishes along its poles. In any
case, $wa$ and $va$ vanish along $\Sigma_{2,3}$, and so does $\varphi(u,v,w)v^2c(w,w/u,w/v)$ if it is holomorphic. Thus, once 
$\J_*V$ is multiplied by a holomorphic function in order to compensate for its poles, 
its second and third coordinates vanish along $\Sigma_{2,3}$. This implies that $\Sigma_{2,3}$ is a leaf of $\J_*\F$ (which includes the case that $\Sigma_{2,3}$
may be contained in $\Sing(\J_*\F)$). \end{proof}

%%%%%%%%%%%
\subsection{No invariant foliation}
%%%%%%%%%%%

In what follows, we assume that there is a foliation $\F$ of dimension $1$ on $X$ that is invariant under the action of $f_X^k$ for some $k\geq 1$. 
We fix a family of local vector fields $({\mathcal{U}}_i, V_i)$ defining $\F$ that satisfy the two properties (i) and (ii), and seek for a contradiction.

%%%%%%%%%%%
\subsubsection{Induced foliation along $\Sigma_0$}
%%%%%%%%%%%
Two cases may occur. Either $\F$ is tangent to $\Sigma_0^X$ at the generic point of $\Sigma_0^X$, or it is generically transverse to $\Sigma_0^X$. 
In this section, we assume that $\F$ is tangent to $\Sigma_0^X$. Thus, $\Sigma_0^X$ is endowed with the codimension $1$ foliation $\F_0$ that is induced by $\F$. 
Local vector fields defining $\F_0$ are obtained by  restricting the $V_i$ to $\Sigma_0$ and then dividing $V_{i\vert \Sigma_0}$ by the equation of the codimension $1$ part of $\{V_i=0\}$ if necessary. 

Since $\F$ is $f_X^k$-invariant, $\F_0$ is $g^k$-invariant; more precisely, the projection of $\F_0$ onto $\Sigma_0$ is $g^k$-invariant. But this is not possible by Theorem~\ref{T:no_g_inv}. Since $f_X$ permutes the four irreducible components of the invariant cycle $\Gamma$, we have proved the following lemma.

\begin{lem}
If $\ell\geq 2$ and $k\neq 0$, any $f_X^k$-invariant foliation $\F$ is transverse to $\Sigma_0^X$, $E_1$, $\Sigma_2^X$, and $E_3$ at the generic point of these four surfaces. 
\end{lem}

In what follows, we assume that the invariant foliation $\F$ is transverse to $\Sigma_0^X$, $E_1$, $\Sigma_2^X$, and $E_3$ at the 
generic point, and we derive a contradiction. For simplicity, the proof is given for $f_X$-invariant foliations, but it applies directly to $f_X^k$-invariant foliations. 

%%%%%%%%%%%
\subsubsection{Behavior along indeterminacies} 
%%%%%%%%%%%
Consider the line $L_1=\Sigma_{0,3}\subset \Sigma_0$. Its ``image'' by $\J$ is the opposite edge $\Sigma_{1,2}$ of the
tetrahedron $\Delta$. Then, its image by the linear projective transformation $\L$ is the line $\L(\Sigma_{1,2})$; it contains
the points $e_1$ and $p_1$, and it can be defined by the equations $x_2=x_3-cx_0=0$. This line is transverse to the plane
$\Sigma_0$ and intersects it in $e_1$. 

Under the action of $f_X^2$ this line $L(\Sigma_{1,2})$, or more precisely its strict transform in $X$, 
is mapped regularly to the curve $L_5\subset \Sigma_0$  defined by the equations 
\[
x_0=0, \quad x_1=(a+c)x_2.
\]

\begin{lem}
The negative orbit of the curve $L_1$ under the action of $g$ is an infinite collection of curves $g^{-n}(L_1)$. None of these curves $g^{-n}(L_1)$, $n\geq 1$, is contained in the indeterminacy locus of $f^k$ for $k\leq 4$. 

Similarly, the positive orbit of the line $L_5$ is an infinite collection of curves, and none of the curves $g^n(L_5)$ is contained in the indeterminacy
locus of $f^{-k}$ for $k\leq 4$.
\end{lem}

\begin{proof}
The first assertion is proved at the very end of Section~\ref{par:surf-W}. 

The second assertion is proved along the same lines, but instead of using the invariant curve $E_1$ of $W$, we use the curve $\beta_0$. The line $L_5$ intersects the curve $\beta_0$ at $q_0=[0:a+c:1:\frac{a}{c}(a+c)]$. The
curve $\beta_0$ is mapped to $\beta_1$, then to $\beta_2$, $\beta_3$, and back to $\beta_0$ under the action of $f$. If one parametrizes 
$\beta_0$ by $[0:t:1:\frac{a}{c}t]$ then $f^4$ acts as $t\mapsto t+(a^2+ac+c^2)/a$ and $q_0$ corresponds to the parameter $t_0= a+c$. 
The orbit of $q_0$ under the action of $f^4$ is the sequence of
points $q_{4n}\in\beta_0$ with parameters $(a+c)+ n(a^2+ac+c^2)/a$. None of them is an indeterminacy point of $f$. 

Similarly, if $(u,v,w)$ are local coordinates of the exceptional divisor $E_1\subset Y$, and $\pi(u,v,w)=[uv:1:uw:u]$ is the blow down map,
then $\beta_1$ is the curve of the exceptional divisor $E_1=\{u=0\}$ parametrized by $(u=0,v=\frac{c}{a}s, w=s)$.
Using this parametrization, the point $q_{4n}$ is mapped by $f$ to the point $q_{4n+1}\in \beta_1$ corresponding to the parameter $s$ defined by
\[
s^{-1}= \frac{a^2+ac+c^2}{c} \cdot \frac{(a/c)^{n+1}-1}{(a/c)-1} .
\]
Since $a/c$ is not a root of unity, none of these points is an indeterminacy point of $f_X$. The proof is similar for
the points $q_{4n+j}=f^j(q_{4n})$, $j=2,3$.
   \end{proof}

Now, consider the foliation $\F$ along the family of curves $g^{-n}(L_1)$, $n\geq 1$. 
The tangency locus of $\F$ with $\Sigma_0^X$ is a Zariski closed set. Therefore, $\F$
is transverse to $\Sigma_0$ along the generic point of $g^{-n}(L_1)$ for large enough $n$. 
Since $g$ and $f$ are regular near the generic point of $g^{-n}(L_1)$ (because $n\geq1$)
and $f$ preserves both $\Sigma_0$ and $\F$, we deduce that $\F$ is transverse to $\Sigma_0$ 
at the generic point of $L_1$. 

\begin{lem}
The curve $\L(\Sigma_{1,2})$ is an algebraic leaf of the foliation $\F$. 
\end{lem}

\begin{proof}
Since the foliation $\F$ is transverse to $L_1=\Sigma_{0,3}$,    Lemma~\ref{par:JonFol} shows that
$\J$ transforms $\F$ into a foliation that is tangent
to $\Sigma_{1,2}$. The conclusion follows from the definition of $f$ as $\L\circ \J$.\end{proof}

%%%%%%%%%%%
\subsubsection{Conclusion} 
%%%%%%%%%%%

The curve $\L(\Sigma_{1,2})$ is a leaf of $\F$ that is contained in $\Sigma_2$. Its image $L_5$ 
under $f^2$ is   a leaf of $\F$ that is contained in $\Sigma_0$, because $f^2$ is regular at the generic point of $\L(\Sigma_{1,2})$. 
Thus,  $L_5$ is contained in the tangent locus of $\F$ and $\Sigma_0$. 
Moreover, $L_5$ is not contained in the exceptional set of $f^n$ and $g^n$, nor in their indeterminacy sets, for $n>0$,
hence the foliation $\F$ is tangent to $\Sigma_0$ along the $g$-orbit of $L_5$.
Since this orbit is infinite, the tangency locus between $\F$ and $\Sigma_0$ contains
infinitely many curves and is not a proper Zariski closed subset of $\Sigma_0$. This 
contradiction shows that $f_X$ does not preserve any foliation of dimension $1$. 

The same strategy applies verbatim to rule out  $f^k_X$-invariant foliations of dimension $1$. 

\begin{pro}\label{pro:no-inv-fol-curves}
If $\ell\geq 2$ and $(a,c)$ satisfies the $\ell$-condition, then $f_X\colon X \dasharrow X$ and its non-trivial iterates
do not preserve any dimension $1$ foliation on $X$.
\end{pro}

%
%%%%%%%%%%%%%%%%%%%%%%%%%%%%%%%%%%%%%%%%%%%%%%%%%%%%%%%%%%%%%%%%%%
%
\section{No invariant foliations: Co-dimension $1$ foliations}\label{par:fol-surfaces} 
%
%%%%%%%%%%%%%%%%%%%%%%%%%%%%%%%%%%%%%%%%%%%%%%%%%%%%%%%%%%%%%%%%%%
%

In this section we prove that  the pseudo-automorphism $f_X$ and its iterates $f_X^k$, $k\neq 0$, do not preserve any foliation of codimension $1$.
We proceed in three steps. First, we exclude the existence of periodic but non invariant foliations. Then, we show that there is no 
invariant algebraic fibration. The third step excludes the existence of $f$-invariant foliation.

%%%%%%%%%%%
\subsection{Foliations of codimension $1$} 
%%%%%%%%%%%

%%%%%
\subsubsection{Local $1$-forms}\label{par:cod1fol}
%%%%%
Let $M$ be a smooth complex projective threefold. 
A codimension $1$ (singular) holomorphic foliation $\G$ of $M$  is determined by a covering ${\U_i}$ of $M$ together with 
holomorphic $1$-forms $\omega_i\in H^{0}(\U_i,\Omega_M^1)$ such that 
\begin{itemize}
\item[(i)]  $\omega_i$ does not vanish in codimension $1$;
\item[(ii)] $\omega_i\wedge d\omega_i=0$ on $\U_i$;
\item[(iii)] $\omega_i=g_{i,j} \omega_j$ on $\U_i\cap \U_j$ for some $g_{i,j}\in {\mathcal{O}}^*(\U_i\cap \U_j)$.
\end{itemize}
The integrability condition $(ii)$ assures that the distribution of planes ${\rm ker}(\omega_x)\subset T_xM$ is integrable: The integral
submanifolds are the {\bf{local leaves}} of the foliation $\G$. Condition $(iii)$ means that the local leaves defined on $\U_i$ by $\omega_i$ patch
together with the local leaves defined on $\U_j$ by $\omega_j$. Condition $(i)$ can always be achieved by dividing $\omega_i$ by 
the equation of its zero set if it vanishes in codimension $1$.

The {\bf{singular locus}} of $\G$ is locally defined by
\[
\Sing(\G)\cap \U_i=\{z\in \U_i \; \vert \; \omega_i(z)=0\}.
\]
This does not depend on the choice of the local defining $1$-form $\omega_i$
because the $g_{i,j}$ do not vanish. Thus, $\Sing(\G)$ is a well defined complex
analytic subset of $X$ of codimension $\geq 2$.

If $h$ is a birational transformation of $X$, one says that $h$ {\bf{preserves}} the foliation $\G$, or that $\G$
is {\bf{invariant}} under the action of $h$,  if $h$  maps local leaves to local leaves. This is equivalent to $h^*\omega_i=\phi_{i,j}\omega_j$ 
on $h^{-1}(\U_i)\cap\U_j$ for some meromorphic functions $\phi_{i,j}$.

%%%%%
\subsubsection{Global $1$-forms and the co-normal bundle}
%%%%%
If $\omega$ is a global meromorphic $1$-form that satisfies $\omega\wedge \omega=0$, then $\omega$ determines a unique
foliation $\G$. In local charts $\U_i$, there are meromorphic functions $h_i$ such that the zeros  (resp. the poles) of $h_i$
coincide with the divisorial part of the zeros of $\omega$ (resp. with the poles); such a function is unique up to multiplication by
a unit $a$, i.e. a holomorphic function that does not vanish. Then, the local $1$-forms $\omega_i=(h_i)^{-1}\omega$ determine
the foliation $\G$ in the sense of \S~\ref{par:cod1fol}.

Consider a foliation $\G$   determined by the family $(\U_i,\omega_i)$ of local $1$-forms. 
The cocycle $(g_{i,j})\in H^{1}(M,{\mathcal{O}}^*_M)$ defined by property $(iii)$ determines a
 line bundle $\mathcal{L}$. Assume that its dual bundle ${\mathcal{L}}^\vee$ has a meromorphic
 section $s$: Locally, $s$ is defined by meromorphic functions $s_i\colon \U_i\to \C$ satisfying
  \[
 s_i=(g_{i,j})^{-1} s_j.
 \]
Define $\omega$ locally by $\omega:=s_i\omega_i$; then, $\omega$ is a global meromorphic
$1$-form that determines $\G$. Thus, the line bundle ${\mathcal{L}}^\vee$
can be identified to the {\bf{co-normal bundle}} of $\G$, and we denote it by $N^*_\G$ in what follows. 

If $\G$ is defined by the global meromorphic $1$-form $\omega$ the divisor
\[
D(\omega)={\mathrm{Zeros}}(\omega)-{\mathrm{Poles}}(\omega)
\]
satisfies 
\[
[D(\omega)]=c_1(N^*_\G)
\]
where $c_1$ denotes the Chern class and $[D]$ the divisor class (viewed in $H^{1,1}(M;\R)$).

\begin{rem}
On a rational variety, the line bundle $N^*_\G$ has a meromorphic section (because
$h^{1,0}(M)=0$). Thus, codimension $1$ foliations are always defined by global meromorphic
$1$-forms.
\end{rem}

%%%%%%%%%%%
\subsection{From dimension $2$ to dimension $1$} 
%%%%%%%%%%%
In this paragraph, we fix an integer $k>1$, and we assume that $f_X^k$ preserves a codimension $1$ foliation $\G$
of $X$ which is not invariant by $f_X^j$ for $1\leq j\leq k-1$. 

The orbit of  $\G$ under the action of $f$ and its iterates form a collection of $k$-distinct foliations
$\G_j=(f^*)^j\G$, $1\leq j\leq k$, with $\G_k=\G$. Since these foliations are distinct, the tangent planes to $\G_{1}$ and
to $\G_2$ are distinct planes in $T_mX$ at the generic point $m\in X$; hence, $T_m\G_1\cap T_mG_2$ is a line $L_{1,2}(m)$
in $T_mX$. These lines form a meromorphic distribution of tangent lines to $X$, this distribution defines a foliation $\F$
of dimension $1$ on $X$, and this foliation is $f^k$-invariant. This remark contradicts the main result of Section~\ref{par:fol-curves} and proves
the following proposition.

\begin{pro}\label{pro:fk-to-f}
Let $\ell$ be an integer with $\ell\geq 2$. Let $f_X:X \dasharrow X$ be the pseudo-automorphism given by 
a couple of parameters $(a,c)$ that satisfies the $\ell$-condition.
If a non-trivial iterate $f_X^k$ of $f_X$ preserves a foliation $\G$, then $f_X$ preserves $\G$.
\end{pro}

%%%%%%%%%%%
\subsection{There is no invariant fibration} 
%%%%%%%%%%%

\begin{pro}\label{pro:no-inv-fib}
Let $\ell$ be an integer with $\ell\geq 2$. 
If the parameter $(a,c)$ satisfies the $\ell$-condition, the automorphism $f_X$  and its iterates $f_X^k$, $k>1$, do not preserve
any invariant fibration. 
\end{pro}

To prove this assertion, we fix a positive integer $k$ 
and suppose that $f_X^k$ preserves a fibration $\varphi\colon X\dasharrow C$, where $C$ is a (smooth) connected
Riemann surface. In other words, $\varphi\circ f_X^k=h\circ \varphi$ for some automorphism $h$ of $C$.
From Proposition~\ref{pro:fk-to-f}, we may -- and do -- assume that
\[
k=1.
\]

%%%
\subsubsection{The linear system determined by $\varphi$}
%%%

 \begin{lem}
The curve $C$ is a projective line $\P^1(\C)$ and the automorphism $h\colon C\to C$ is conjugate to a similitude $\zeta \mapsto \kappa \zeta$ or 
to the translation $\zeta\mapsto \zeta +1$. 
\end{lem}
\begin{proof}
Since $X$ is rational, the curve $C$ is covered by a rational curve; hence, $C$ is isomorphic to $\P^1(\C)$. The automorphism $h$ of $C$ is a M\"obius transformation of
the Riemann sphere and, as such, is conjugate to a similitude or a translation.
\end{proof}

The fibers of $\varphi$ determine a linear system of hypersurfaces  $S(c)=\varphi^{-1}(c)$ in $X$, $c\in C$. 
From Section~\ref{par:InvHyper}, we deduce that there is a positive integer $r$ such that the divisor class
$[S(c)]$ satisfies $[S(c)]=r[\Gamma]$ for all $c$. 

Viewed on $\P^3_\C$, the pencil of hypersurfaces $(S(c))_{c\in C}$ is generated by the two surfaces $S(0)$ and $S(\infty)$. Changing the coordinate of $C$ 
if necessary, one may assume that $S(\infty)$ is the hypersurface $\Sigma_0+\Sigma_2$, with multiplicity $r$; its equation is 
\[
Q=(x_0x_2)^r.
\] 
The hypersurface $S(0)$   is also given by a homogeneous polynomial $P$ of degree $r$. The invariance of the pencil reads
\[
f^*Q={\rm{Jac}}(f)^r Q,  \; {\text{and}}\;  f^*P= {\rm{Jac}}(f)^r \kappa P \; \quad {\text{or}}\ \ \ \;  f^*P={\rm{Jac}}(f)^r (P +Q).
\]
according to the two distinct possibilities for $h$. (here $f^*P$ denotes the composition of $P$ with a lift $F$ of 
$f$ to $\C^4$, as in Section~\ref{par:InvHyper}, and ${\rm{Jac}}(f)$ stands for ${\rm{Jac}}(F)$)

The base points of the linear system $(S(c))_{c\in C}$ form a curve of $X$. This curve $B$ is contained in all
members of the linear system; hence, $B$ is contained in $\Gamma$. Its trace on $\Sigma_0$ is a $g$-invariant curve.

%%%
\subsubsection{Conclusion by computation}
%%%

To study the equation $P$ of $S(0)$ in $\P^3$, note that Lemma~\ref{lem:gw-inv} implies that the equation $P(0,x_1,x_2,x_3)=0$
defines a divisor $B_{\vert \Sigma_0}\subset \Sigma_0$ of the form $m_1 \Sigma_{0,2}+m_2 \beta_0$; hence 
\[
P(0,x_1,x_2,x_3)=c^{st} x_2^{m_1}(ax_1-cx_3)^{m_2}.
\]
The same argument, once applied to the plane $\Sigma_2$, shows that 
\[
P(x_0,x_1,0,x_3)=c^{st} x_0^{m'_1}(ax_1-cx_3)^{m_2'}.
\]
Thus, 
\[
P = c_1 x_2^{m_1}(ax_1-cx_3)^{m_2} + c_2 x_0^{m'_1}(ax_1-cx_3)^{m_2'} + x_0x_2 R 
\]
for some homogeneous polynomial $R(x_0,x_1,x_2,x_3)$ of degree $\deg(P)-2$.
The integers $m_1,m_2$ and $m'_1, m'_2$ satisfy $m_1+m_2=m'_1+m'_2=\deg(P)=2r$. On the other
hand, $P$ vanishes at $e_i$ with multiplicity $r$, so that the degree of $x_i$ in each monomial of $P$
is at most $r$. Thus, $m_1=m_2=m'_1=m'_2=r$ and we get:  
\begin{lem}
The homogeneous polynomial $P$ can be written 
\[
P = c_1 x_2^{r}(ax_1-cx_3)^{r} + c_2 x_0^{r}(ax_1-cx_3)^{r} + x_0x_2 R 
\]
for some non-zero constants $c_1$ and $c_2$ and some homogeneous polynomial $R$
of degree $2r-2$. 
\end{lem}

We shall now derive a contradiction.
If $P'=\kappa P + \mu Q$ with $\kappa\in \C$ and $\mu\in \C$, then
\begin{eqnarray*}
P'(0,x_1,x_2,x_3) & = & \kappa c_1 x_2^r(ax_1-cx_3)^r\\
 P'(x_0,x_1,0, x_3) & = & \kappa c_2 x_0^r(cx_1-ax_3)^{r}.
\end{eqnarray*}
In particular $P'(0,x_1,x_2,ax_1/c)=0$ and $P'(x_0,x_1,0, cx_1/a)=0$.
But $f^*P={\rm{Jac}}(f)^r (\kappa P + \mu Q)$ for some complex numbers $\kappa$ and $\mu$, and at the same
time it is equal to 
\[
{\rm{Jac}}(f)^r \left\{ c_1(cx_0-ax_2)^rx_3^r + c_2 (ax_0-cx_2)^rx_1^r + x_0x_2 R'\right\}
\]
for some homogeneous polynomial $R'$. Evaluating along $(0,x_1,x_2,ax_1/c)$ and $(x_0,x_1,0, cx_1/a)$
one must obtain $0$; hence,  
\[
c_1a^{2r}+c_2 c^{2r}=0 = c_1c^{2r}+ c_2a^{2r}
\]
and we deduce that $a^{2r}=-c^{2r}$ because $c_1$ and $c_2$ are not equal to zero. This is in contradiction 
with the $\ell$-condition and this contradiction concludes the proof of Proposition~\ref{pro:no-inv-fib}.

%%%%%%%%%%%
\subsection{From invariant foliations to invariant global $1$-forms} 
%%%%%%%%%%%

Let $(a,c)$ satisfy the $\ell$-condition for some integer $\ell\geq 2$, and $f\colon X\dasharrow X$
be the pseudo-automorphism constructed in Section~\ref{par:Part2}. 
Let $\G$ be a codimension $1$ foliation of  $X$ which is invariant under the action of $f$. 

Since $X$ is a rational variety, there exists  a global meromorphic $1$-form $\omega$ that defines  
$\G$. The pull-back of $\omega$ by $f$ is another meromorphic $1$-form defining $\G$. Hence, 
the divisor classes of $D(\omega)$ and $D(f^*\omega)$ coincide with the Chern class of the
co-normal bundle $N^*_\G$. Since $f_X$ is a pseudo-automorphism $[D(f_X^*\omega)]=f_X^*[D(\omega)]$, 
and 
\[
[D(\omega)]=c_1(N^*_\G)=f_X^{*}[D(\omega)].
\]
In other words, the Chern class of the co-normal bundle is $f_X^*$-invariant. This provides a
fixed vector in $H^{1,1}(X;\R)\cap H^{2}(X;\Z)$. 

The eigenspace of $f_X^*$ corresponding to the eigenvalue $1$ is the line $\R [\Gamma]$  in $H^{1,1}(X;\R)$. As a consequence, there is a global meromorphic
$2$-form $\alpha$ on $X$ that does not vanish, has poles along $\Gamma$ (of order $k$ for some $k>0$), and defines the foliation $\G$.
If $\beta$ is another global meromorphic $2$-form with the same properties, then $\beta$ is a multiple of $\alpha$
by a meromorphic function $\varphi$, because they both define the same foliation; since $\beta$ and $\alpha$
define the same divisor $\Gamma$, $\varphi$ has no pole and no zero: Thus, $\varphi$ is a non-zero constant. 
This argument proves the following result. 

\begin{lem}
The Chern class $c_1(N^*_\G)$ is a negative multiple $-k[\Gamma]$ of the class of the invariant cycle $\Gamma$. The
space ${\mathcal{M}}^1_\G(-k\Gamma)$ of meromorphic $1$-forms $\alpha$ on $X$ with $D(\alpha)=-k\Gamma$ is a complex vector space
of dimension $1$.
\end{lem}

Let $\alpha_\G$ be a non-zero element of ${\mathcal{M}}^1_\G(-k\Gamma)$.
Being a pseudo-automorphism, the transformation $f$ acts by pull-back on ${\mathcal{M}}^1_\G(-k\Gamma)$. Thus, {\sl{there is a non-zero
complex number $\eta$ such that 
\[
f_X^*\alpha_\G=\eta\cdot \alpha_\G.
\]
}}
All we need to do now is to exclude the existence of such an $f_X$-invariant, integrable $1$-form.

\vfill\eject

%%%%%%%%%%%
\subsection{Invariant $1$-forms} 
%%%%%%%%%%%

%%%%
\subsubsection{$\alpha_G$ is closed}
%%%%
If $d\alpha_\G\neq 0$, there is a meromorphic vector field $V_\G$ such that 
\[
d\alpha_\G=\iota_{V_\G} \Omega
\]
where $\iota_\xi$ denotes the interior product with $\xi$, and $\Omega$ is the $f$-invariant $3$-form
described in Section~\ref{par:InvHyper}. Since $\Omega$ and $\alpha_\G$ are multiplied by non-zero complex numbers under the action of $f$, 
$V_\G$ is also multiplied by some non-zero complex number. Thus, the foliation determined by $V_\G$
is $f_X$-invariant. This is a contradiction with Proposition~\ref{pro:no-inv-fol-curves} of Section~\ref{par:fol-curves}, 
so that the following lemma is proved. 

\begin{lem}
The  $1$-form $\alpha_\G$ is closed. 
\end{lem}

\begin{rem}
The non-existence of $1$-dimensional foliations is not necessary to prove this
lemma. Indeed, $V_\G$ has no poles and vanishes only along $\Gamma$. Thus, if one projects $V_\G$
onto $\P^3$, one gets a vector field that vanishes along $\Sigma_0\cup \Sigma_2$. It is easy to exclude
the existence of such an $f$-invariant vector field.
\end{rem}

%%%%
\subsubsection{First integral}
%%%%

\begin{lem}
The poles of $\alpha_\G$ along $\Gamma$ have order $k\geq 2$. 
\end{lem}

\begin{proof}
Consider the set of lines through the point $e_1$ in $\P^3$. When lifted to $X$, these lines intersect
$E_1$ but do not intersect the other irreducible components of $\Gamma$, except for those lines that are contained in $\Gamma$. 
Thus, when one restricts $\alpha_G$ to such a curve, one gets a $1$-form on $\P^1_\C$, with a unique pole. 
If the form does not vanish, this pole must have multiplicity $\geq 2$. 

Apply the same argument for the family of lines through $e_3$. Either $\alpha_\G$ vanishes identically 
along this family of lines, or the order of the pole of $\alpha_\G$ along $E_3$ is $\geq 2$. 

This argument concludes the proof if $\alpha_\G$ does not vanish identically along one of these two families of lines. 
If it does, the foliation defined by $\alpha_\G$ is the pencil of planes containing the curve $\Sigma_{0,2}$. 
It is easy to see that $f$ does not preserve any pencil of planes (and this is a consequence of Proposition~\ref{pro:no-inv-fib}).
Thus, $k\geq 2$.
\end{proof}

Since $\alpha_\G$ is closed, $\alpha_\G$ is locally the differential of a function. Locally, $\alpha_\G=d\varphi$, 
where $\varphi$ is obtained by integration of $\alpha_\G$:
\[
\varphi(z)-\varphi(z_0)=\int_{\gamma} \alpha_\G
\]
where $\gamma$ is a local path from $z_0$ to $z$. The local function $\varphi$ is a well defined meromorphic function on any 
simply connected subset $\U$ of $X\setminus \Gamma$. 

To show that $\varphi$ is also locally defined around $\Gamma$, one needs to show that the residue of $\alpha_\G$ along
$\Gamma$ vanishes. Let $\U$ be a small ball that intersects $\Sigma_0^X$, $V$ be a disk in $\U$ that intersects 
$\Sigma_0^X$ transversally, and $\beta$ be a loop in $V$ that turns once around $V\cap \Gamma$. One wants to show that
\begin{equation}\label{residuezero}
\int_\beta \alpha_\G = 0.
\end{equation}
For $V$, one can take the intersection of the strict transform of a line $L\subset \P^3$ with $\U$. The line $L$ and the 
point $p_3$ generate a plane $\Pi\subset \P^3$. Let $\Pi^X$ be the strict transform of $\Pi$ in $X$; it intersects $P_3$ on a line $\Pi^X\cap P_3$ 
and this line intersects $\Sigma_0^X\cap P_3$ in a unique point $z$. Now, one can deform the loop $\beta$ within $\Pi^X$ to a loop 
with the same base point that go straight to $P_3$, makes one turn around $z$ in the line $\Pi^X\cap P_3$, 
and come back to the base point along the same way. But the line $\Pi^X\cap P_3$ is a sphere that intersects $\Gamma$ at a unique
point, namely $z$. Thus, one can  change $\beta$ continuously to a trivial loop, and we obtain Equality~\eqref{residuezero}.

As a consequence, the local function $\varphi$ is a well defined meromorphic function on any 
simply connected subset $\U$ of $X$, even if $\U$ intersects $\Gamma$; moreover, $\varphi\colon \U\to \C$ is unique up to an additive constant. 
Thus, there is an open cover of $X$ by simply connected domains $\U_i$ and meromorphic functions $\varphi_i\colon \U_i\to \C$ such that
\begin{itemize}
\item $\alpha_\G=d\varphi_i$ in $\U_i$; 
\item  the intersections $\U_i\cap U_j$ are connected (or empty);
\item  the functions $\varphi_i$ satisfy a cocycle relation:
$\varphi_i=\varphi_j + \delta_{i,j}$
on $\U_i\cap\U_j$ for some complex numbers $\delta_{i,j}$. 
\end{itemize}
The cocycle $(\delta_{i,j})$ determines a class in $H^1(X;\C)$ for Cech cohomology. But $X$ is obtained from 
the projective space $\P^3_\C$ by a finite sequence of blow-ups and, as such, is simply connected. It follows that
the cohomology class of the cocycle $(\delta_{i,j})$ is equal to zero: One can choose the functions $\varphi_i$ in such 
a way that they glue together to define a global meromorphic function 
\[
\varphi_\G\colon X\dasharrow \C.
\]
The level sets of $\varphi_\G$ form a meromorphic fibration of $X$ by algebraic surfaces, the fibers of which are the 
Zariski closure of the leaves of $\G$. Thus, $f$ preserves a fibration, in contradiction with Proposition~\ref{pro:no-inv-fib}. 
This contradiction shows that there is no invariant foliation of codimension $1$ in $X$, and this concludes the proof of
Theorem~\ref{thm:main}.

{\small{
%
%%%%%%%%%%%%%%%%%%%%%%%%%%%%%%%%%%%%%%%%%%%%%%%%%%%%%%%%%%%%%%%%%%
%
\section{Appendix I: a typical computation}\label{par:Appendix} 
%
%%%%%%%%%%%%%%%%%%%%%%%%%%%%%%%%%%%%%%%%%%%%%%%%%%%%%%%%%%%%%%%%%%
%

In this appendix, we describe the computation that leads to Lemma~\ref{lem:l-cond} and Proposition~\ref{pro:picard-action}.

$\bullet$ The curve $\beta_2$ is parametrized by $t\mapsto [1:t:0:ct/a]$. We have 
\[
f[x_0:\;x_1:\;x_2:\;x_3] = [x_0x_1x_2:\;  x_1x_2x_3 + ax_0x_1x_2 :\; x_0 x_2 x_3 :\; x_0 x_1 x_3 + c x_0 x_1 x_2].
\]
In good local coordinates $(u,v,w)$, the projection $\pi\colon Y\to \P^2_\C$ is given by 
\[
\pi(u,v,w)=[u:\;uv:\;uw:\; 1],
\]
with the exceptional divisor $E_3^Y=\{u=0\}$ that is mapped to the point $e_3=[0:\; 0:\; 0:\;1]$. 
Consider  $f$ as a map from the open set $x_0=1$ to the open set $x_3=1$. One gets
\[
f[1:\;x_1:\;x_2:\;x_3]=[\frac{x_2}{x_3+cx_2}:\; \frac{x_2}{x_3+cx_2}(x_3+a) :\; \frac{x_2}{x_3+cx_2} \frac{x_3}{x_1}:\; 1]. 
\] 
Hence, $f_Y$ maps $\Sigma_2^Y$ to $E_3^Y$, and in local affine coordinates $(x_1,x_2,x_3)$ and $(u,v,w)$ one has 
\[
f_Y(x_1,\;x_2,\;x_3) = (\frac{x_2}{x_3+cx_2}, x_3+a, \frac{x_3}{x_1}).
\]
Applied to the parametrization of $\beta_2$, one gets 
\[
t\mapsto [1:t:0:ct/a] \mapsto (u',v',w')=(0, (ct+a^2)/a, c/a).
\]
The image is a curve of $E_3^Y$ (because $u'=0$) that satisfies $ax_2=cx_0$ (because $w'=c/a$).

\vspace{0.2cm}

$\bullet$ Then, consider the restriction of $f_Y$ to $E_3^Y$. For this, compute $f\circ \pi$ in the $(u,v,w)$
coordinates, to get
\begin{eqnarray*}
f\circ\pi(u,v,w)  &  = & [u^3vw :\; u^2vw + a u^3vw :\;  u^2w :\;  u^2v + cu^3vw]\\
& = & [uvw :\; vw + a uvw :\;  w :\;  v + cuvw]
\end{eqnarray*}
after division by the common factor $u^2$. In particular, for $u=0$, one gets 
\[
f\circ \pi(0,v,w)= [0:\; vw :\; w :\; v].
\]
This implies that $E_3^Y$ is mapped to $\Sigma_0$ and, moreover, that $E_3^Y$ appears with 
multiplicity $2$ in the pull-back of $\Sigma_0$ by $f_Y$ (because we had to divide by $u^2$). 
This proves two of our statements, namely $f_Y(E_3^Y)=\Sigma_0^Y$, and the coefficient $-2$
in front of $\hat{E_3}$ in the formula for $f_X^* H$. 
Then, the curve $f_Y(\beta_2)$ is parametrized by 
\[
t\mapsto [0:(ct+a^2)/a : 1: (ct+a^2)/c]. 
\]
As mentioned, this curve is $\beta_0$ (i.e. $ax_1=cx_3$ in $\Sigma_0$).

$\bullet$ One then pursues this kind of computation, by blowing up $e_1$ into $E_1^Y$. This 
gives $\pi(u,v,w)=[u:1:uv:uw]$ in local coordinates and 
\[
f[x_0:\; x_1 :\; x_2 :\;x_3 ] = [\frac{x_0}{x_3+ax_0}:\; 1 :\; \frac{x_0}{x_3+ax_0}\frac{x_3}{x_1}:\;\frac{x_0}{x_3+ax_0} \frac{x_3+cx_2}{x_2}]
\]
i.e. $(u,v,w)= (x_0/(x_3+ax_0), x_3/x_1, (x_3+cx_2)/x_2)$. Along the curve $\beta_0$ one gets the parametrization 
\[
t\mapsto (u',v',w')=(0, a/c, t+(a^2+c^2)/c).
\]

$\bullet$ Then, with coordinates $(u,v,w)$ near $E_1^Y$, one obtains 
\[
f\circ \pi(u,v,w)=[v:\; vw+av :\; uvw :\; w+cv]
\]
after division by the common factor $u^2$. Hence, applied to the curve $\beta_1$, one gets
the new parametrization 
\[
t\mapsto [ 1 :\; t + \frac{a^2}{c} + c+ a :\; 0:\; \frac{c}{a}(t + \frac{a^2}{c} + c+ a)]
\]
of the curve $\beta_2$. This proves that $h=f^4_Y$ transforms the parameter $t$ into
\[
t'=t + \frac{a^2}{c} + c+ a,
\]
as explained before Lemma~\ref{lem:l-cond}.

}}

%
%%%%%%%%%%%%%%%%%%%%%%%%%%%%%%%%%%%%%%%%%%%%%%%%%%%%%%%%%%%%%%%%%%
%

{\small{

%
%%%%%%%%%%%%%%%%%%%%%%%%%%%%%%%%%%%%%%%%%%%%%%%%%%%%%%%%%%%%%%%%%%
%
\section{Appendix II: Tori of dimension 3}\label{par:Part1} 
%
%%%%%%%%%%%%%%%%%%%%%%%%%%%%%%%%%%%%%%%%%%%%%%%%%%%%%%%%%%%%%%%%%%
%

In this section, we prove Propositions 7.3 and 7.5.       %~\ref{pro:Tori}.

\subsection{Automorphisms of tori} Let $A$ be a compact complex torus of dimension $d$. 
Denote by $\pi\colon V\to A$ the universal cover of $A$; here
$V$ is a complex vector space of dimension $d$ and $A=V/\Lambda$ for some
lattice $\Lambda\subset V$, and $\pi$ is the projection $V\to V/\Lambda$.
Let $f\colon A\to A$ be an automorphism of $A$. 
It lifts to an affine transformation $F$ of $V$, so that $\pi\circ F=f\circ\pi$.

If one composes $f$ with a translation, which does not change the action of $f$ on the cohomology
of $A$, one may assume that $F$ is linear. Since $F$ preserves $\Lambda$, its determinant has
modulus $1$ (since $\vert \det(f)\vert^2$ is the topological degree of $f$).

\subsection{Dimension $2$}

\subsubsection{} Assume that the dimension $d$ of the torus $A$ is equal to $2$. If the automorphism $f$ is not 
cohomologically hyperbolic, then $\lambda_1(f)$ is equal to $1$, because both $\lambda_0(f)$ and
its topological degree $\lambda_2(f)$ are equal to $1$. 

\begin{lem}\label{lem:Kro}
Let $g$ be an automorphism of a compact k\"ahler manifold $M$.
If $\lambda_1(f)=1$, then all eigenvalues of $f^*\colon H^*(M;\C)\to H^*(M;\C)$ are roots of unity.
\end{lem}

\begin{proof}
If $\lambda_1(f)$ is equal to $1$, then all eigenvalues of $f^*$ have modulus $\leq 1$ (see \cite{}). On the
other hand, they are algebraic integers, because $f^*$ preserves the lattice $H^*(M;\Z)$ of $H^*(M;\Z)$. 
By Kronecker Lemma, all eigenvalues are roots of $1$.
\end{proof}

\subsubsection{}\label{par:torus_d2} Thus, if $f$ is an automorphism of a $2$-dimensional torus, either $\lambda_1(f)>1$, (and thus $f$ is
cohomologically hyperbolic), or all eigenvalues of $f^*$ are roots of unity. Changing $f$ into an iterate, 
all eigenvalues of $f^*$ are equal to $1$; since $H^{1,0}(A;\C)$ identifies with the dual $V^\vee$ of $V$, 
there is a basis of $V$ in which the matrix of $F$ is upper triangular, with its two diagonal coefficients 
equal to $1$. Assume that $F$ is not the identity. Then the first vector of that basis determines a line 
$L_F\subset V$ which is defined over the rational numbers with respect to the lattice $\Lambda$; in
other words, $L_F$ intersect $\Lambda$ over a co-compact lattice. The projection of $L_F$ into
$A$ is an elliptic curve, and it $f$-invariant. Thus, $L_F/\Lambda$ and its translates define an $f$-invariant
elliptic fibration. We have proved the following fact:

\vspace{0.2cm}

{\sl{If $f$ is an automorphism of $2$-dimensional torus $A$, and $f$ is not cohomologically hyperbolic, then 
$f$ preserves an elliptic fibration. There is an elliptic curve $B$ and a surjective morphism $\pi\colon A\to B$
such that (i) the fibers of $\pi$ are tori of dimension $1$ in $A$, and (ii) there is an automorphism ${\overline{f}}$
of $B$ such that $\pi \circ f ={\overline{f}}\circ \pi$. }}

\subsection{Dimension 3}

Let us now assume that $A$ has dimension $3$. 

\subsubsection{} 
Consider the action of $f$ on the cohomology of $A$. 
\begin{enumerate}
\item $H^1(A;\Z)$ is isomorphic to the dual $\Lambda^\vee$ of $\Lambda$ and the action of
$f$ on this space is given by the action of $F$ on $\Lambda^\vee$;
\item on $H^{1,0}(A,\C)$ the action of $f$ is given by the action of $F$ on  $V^\vee$, which 
we denote by $F^*\colon V^\vee \to V^\vee$;
\end{enumerate}

Denote by $\alpha$, $\beta$, and $\gamma$ the three eigenvalues of $F^*$ (repeated according
to their multiplicity), with 
\[
\vert \alpha\vert \geq \vert \beta \vert \geq \vert \gamma\vert.
\]
The eigenvalues of $f^*$ on $H^{1,1}(A,\C)$ are 
\[
\alpha {\overline{\alpha}}, \,\; \alpha{\overline{\beta}}, \,\; \beta  {\overline{\alpha}}, \,\;  \beta  {\overline{\beta}},  \,\; \alpha  {\overline{\gamma}},
\,\;  \gamma  {\overline{\alpha}}, \,\;  \beta  {\overline{\gamma}},  \,\; \gamma  {\overline{\beta}}, \,\;  \gamma  {\overline{\gamma}}.
\]
In particular, $\alpha {\overline{\alpha}}$ is the largest eigenvalue and, as such, coincides with $\lambda_1(f)$. On $H^{2,2}(A,\C)$, 
the largest eigenvalue is $\alpha{\overline{\alpha}} \beta  {\overline{\beta}}= (\gamma {\overline{\gamma}} )^{-1}$. 

Thus, 
$\lambda_1(f)=\lambda_2(f)$ if and only if 
\[
\vert \beta \vert =1.
\] 
From now on, we assume that this property is satisfied by $f$.
Thus, $F$ is an element of $\GL(V)$ with three eigenvalues $\alpha$, $\beta$, $\gamma$ that satisfy
\[
\vert \alpha\vert \geq \vert \beta  \vert =1 \geq \vert \gamma\vert,
\]
\[
\vert \alpha \vert =\vert \gamma \vert^{-1}
\]
Viewed as an element of $\GL(\Lambda^\vee)$, $F$ corresponds to a $6\times 6$ matrix with integer coefficients, 
whose eigenvalues are $\alpha$, $\beta$, $\gamma$, and their complex conjugates. Denote by $\chi(t)$ the characteristic
polynomial of this matrix; it is an unitary polynomial of degree $6$ with integer coefficients. (moreover, changing $f$ 
in an iterate, the determinant, i.e. the constant coefficient of $\chi$, is $1$)

\subsubsection{Dynamical degree} 
Assume that $\lambda_1(f)$ is equal to $1$. By Lemma~\ref{lem:Kro}, 
the three eigenvalues $\alpha$, $\beta$ and $\gamma$ are roots of unity. As in Section~\ref{par:torus_d2}, we prove that
$f$ preserves a fibration of $A$ by translates of subtori. 

Thus,  we assume in what follows that $\vert \alpha\vert>1$. This implies 
\[
\vert \alpha\vert>1, \, \; {\text{ and }} \; \, \vert \gamma\vert<1.
\]

\subsubsection{Invariant fibrations} Assume that $f$ preserves an invariant fibration $\pi\colon A \to B$. Then $B$
is a torus of dimension $1$ or $2$, and there is an automorphism ${\overline{f}}$ of $B$ such that $\pi\circ f= {\overline{f}}\circ \pi$.

If $\dim(B)=1$, then ${\overline{f}}^{	12}$ acts trivially on $H^{1,0}(B;\C)$. This implies that $\beta^{12}=1$. 
If $\dim(B)=2$, the fiber of $\pi$ containing the origin is an $f$-invariant elliptic curve $E\subset A$. The action of
$f^{12}$ on this curve is the identity. Hence, again, $\beta^{12}=1$. We have proven the following lemma. 

\begin{lem}\label{lem:roots_of_1} Let $f$ be an automorphism of a compact complex torus of dimension $3$. Then $f$ preserves a fibration if and only if
$\beta$ is a root of unity (of order at most $12$ if $\lambda_1(f)>1$).
\end{lem}

Let us add one comment. If $f$ preserves a fibration with $\dim(B)=2$, then $\lambda_1(f)=\lambda_1({\overline{f}})$ is the first
dynamical degree of an automorphism of a surface and, as such, is either equal to $1$, to a quadratic number, or to a Salem
number (of degree $4$ because $\dim(H^{1,1}(B))=4$). If $\dim(B)=1$, then $f$ induces an automorphism of the fiber $E$
of $\pi$ that contains the origin; again, one obtains that $\lambda_1(f)$ is  equal to $1$, to a quadratic number, or to a Salem
number of degree $4$.

\subsubsection{Degree of $\beta$} 

Denote by $\varphi(t)\in \Z[t]$ the minimal polynomial  of $\beta$ and by $d(\beta)$ its  degree.  Since ${\overline{\beta}}=\beta^{-1}$,
either $\varphi$ has degree $1$, and $\beta =1 $ or $-1$, or $d(\beta)$ is even. 
\begin{itemize}
\item If $d(\beta)=1$, then $\beta =1$ or $-1$.
\item If $d(\beta)=2$, then $\beta$ and $\overline{\beta}$ are the two roots of $\varphi$; hence $\beta$
is a quadratic root of $1$.
\item If $d(\beta)=4$, write $\chi(t)=\varphi(t)\psi(t)$ with $\psi(t)\in \Z[t]$. If $\alpha$  is a root of $
\psi(t)$, then ${\overline{\alpha}}$ also, and the roots of $\varphi$ are $\beta$, $\gamma$, and
their conjugates: By Kronecker Lemma, this would imply that $\gamma$ is a root of $1$, a contradiction.
A similar contradiction is easily drawn if $\gamma$ is a root of $\psi(t)$.
\end{itemize}
Thus, {\sl{either $d(\beta)=6$ or $\beta$ is a root of one of order $\leq 4$}}. In the latter case, an iterate of $F$
has an eigenvalue equal to $1$; again, this implies that $f$ preserves a fibration of $A$ by subtori.

\subsubsection{} Let us assume now that $\lambda_1(f)$ is a Salem number. As remarked by Tuyen Truong, 
this implies that $\lambda_1(f)=\lambda_2(f)$. Truong's proof, which works for automorphisms $f\colon M\to M$ 
of compact k\"ahler manifolds of dimension $3$, is as follows:
\begin{proof}
The duality between $H^{1,1}(M;\C)$ and $H^{2,2}(M;\C)$ gives
\[
\lambda_2 ( f )  =  \lambda_1 ( f^{-1} ).
\]
Moreover since $\lambda_1(f)$  is  a Salem number, any eigenvalue for  $f^*_1$  will also be an eigenvalue of  $(f^{-1})^*_1$.  Thus we have an inequality $\lambda_1 ( f )  \leq  \lambda_2 (f )$. Arguing with  $f^{-1}$  gives equality.
\end{proof}
Coming back to automorphisms of tori of dimension $3$ with $\lambda_1(f)=\alpha{\overline{\alpha}}$ a Salem number, 
we deduce that  $\vert \beta\vert =1$ and $\gamma{\overline{\gamma}}$ is the only Galois conjugate of $\lambda_1(f)$ 
with modulus $<1$.

Let us assume that $d(\beta)=6$. Let $\sigma$ be an automorphism of the field ${\overline{\Q}}$ that
maps $\alpha$ onto $\beta$; then $\sigma(\alpha{\overline{\alpha}})= \beta \sigma({\overline{\alpha}})$
is a Galois conjugate of $\lambda_1(f)$. If $\sigma(\alpha{\overline{\alpha}})$ has modulus
$1$, then $\sigma({\overline{\alpha}})$ too;  hence $\sigma({\overline{\alpha}})={\overline{\beta}}$ and
$\sigma(\lambda_1(f))=1$, a contradiction. We deduce that
the modulus of $\sigma(\alpha{\overline{\alpha}})$ is $<1$ and $\sigma(\alpha{\overline{\alpha}})=\gamma {\overline{\gamma}}$.
This is a contradiction because $\sigma({\overline{\alpha}})$ is a conjugate of $\alpha$, and no conjugate of $\alpha$
has modulus equal to $\gamma {\overline{\gamma}}$.

Hence $d(\beta)<6$ and we know from the previous paragraph that $f$ preserves an invariant fibration. Thus, with Lemma~\ref{lem:roots_of_1}, we proved the
following statement of Oguiso and Truong (their proof is almost the same):

\begin{pro}
Let $f\colon A\to A$ be an automorphism of a $3$-dimensional torus with $\lambda_1(f)>1$. Then, $\lambda_1(f)$ is a quadratic number or a Salem
number if and only if $f$ preserves a non-trivial fibration of $A$.
\end{pro}

\subsubsection{} 
We now assume that $\lambda_1(f)=\lambda_2(f)$ (i.e. $\vert \beta \vert =1$) and $d(\beta)=6$. In particular, the characteristic polynomial $\chi$ is irreducible. 

\begin{rem}
Since $\beta^{-1}={\overline{\beta}}$ is a conjugate of $\beta$, and the Galois group acts
by permutations on the set of roots of $\chi$, $\alpha^{-1}$ and $\gamma^{-1}$ are roots
of $\chi$. This implies that $\alpha^{-1}=\gamma$ or ${\overline{\gamma}}$. We deduce
that the vector space $W_1(f^*)$ of eigenvectors of $f^*$ in $H^2(A;\C)$ has dimension 
$3$. \end{rem}

Thus, if $d(\beta)=6$, then the roots of $\chi(t)$ can be organized as follows: two complex
conjugate roots with modulus $>1$, which we denote by $\alpha$ and ${\overline{\alpha}}$, 
two on the unit circle, namely $\beta$ and ${\overline{\beta}}$, and two in the open
unit disk, $\gamma$ and ${\overline{\gamma}}$ such that
\[
\{\alpha, {\overline{\alpha}} \}=\{\gamma^{-1}, {\overline{\gamma}}^{-1}\}.
\]
Thus, 
\[
\chi(t)= \prod (t-\nu)(t-1/\nu)
\]
where $\nu$ describes $\{\alpha, \beta, {\overline{\alpha}}\}$. In other words, $\chi(t)$ is
a reciprocal polynomial: $t^6\chi(t^{-1})=\chi(t)$. 

\begin{pro}
There exists an automorphism $f$ of a compact torus of dimension $3$ such that (i)
$\lambda_1(f)=\lambda_2(f)>1$ and (ii) $f$ does not preserve any non-trivial fibration. 
\end{pro}

\begin{proof}
Consider an irreducible, unitary polynomial $\chi(t)\in \Z[t]$ of degree $6$, with the additional
property $\chi(1/t)t^6=\chi(t).$ Assume that one of its roots $\alpha\in \C$ has modulus $>1$ and  
is not a real number and that another root $\beta$ has modulus $1$. Then the roots of $\chi$
are $\alpha$, $\gamma=1/{\overline{\alpha}}$, $\beta$, and their complex conjugates. Write
\[
\chi(t)= t^6 + at^5 + bt^4 + ct^3 + dt^3+et+ 1.
\]
Then, consider the companion matrix $M_\chi\in \SL_6(\Z)$ with characteristic polynomial $\chi(t)$. 
Its eigenvalues are $\alpha$, $\gamma=1/{\overline{\alpha}}$, $\beta$, and their complex conjugates.
Thus, there is a complex structure ${\sc{j}}$ on $\R^6$ which commutes to the action of $M_\chi$. 
The quotient of this complex vector space $(\R^6,{\sc{j}})$ by the $M_\chi$-invariant lattice $\Z^6\subset \R^6$
is a complex torus, on which $M_\chi$ induces an automorphism $f$ with all the desired properties. 

It remains to construct such a polynomial $\chi(t)$. One can take 
\[
\chi(t)=t^6+ t^5 -2t^3 + t + 1.
\]
See the remark below for a general strategy.
\end{proof}

\begin{rem}
For the proof of the previous proposition, one needs to construct irreducible, unitary polynomials
$\chi(t)\in \Z[t]$ such that: 
\begin{itemize}
\item $\chi$ is reciprocal and has exactly two roots on the unit circle;
\item $\chi$ has degree $6$.
\end{itemize}
Write 
\[
\chi(t)= t^6 + at^5 + bt^4 + ct^3 + dt^3+et+ 1.
\]
Consider the new variable $s=t+1/t$. Then 
\[
\theta(s):=\chi(t)/t^3= s^3+as^2 + (b-3)s + (c-2a)
\]
is a polynomial of degree $3$, with integer coefficients, and exactly one real 
root (namely $s=\beta+ {\overline{\beta}}$); this root is between $-2$ and $2$. 
Then, it is not hard to start from $\theta$ with the desired properties and reconstruct
$\chi$.
\end{rem}
}}

%
%%%%%%%%%%%%%%%%%%%%%%%%%%%%%%%%%%%%%%%%%%%%%%%%%%%%%%%%%%%%%%%%%%
%

\vspace{8mm}

\bibliographystyle{plain}
\bibliography{bibliobkc}
\nocite{}

\end{document}